\documentclass[11pt]{article}

\usepackage{amsmath,amssymb,amscd,amsthm}
\usepackage{ascmac}
\usepackage[all]{xy}
\usepackage{enumerate}
\usepackage{cancel}
\usepackage[dvips]{graphicx}
\usepackage{authblk}
\usepackage{comment}
\usepackage{bm}
\usepackage{mathrsfs}
\usepackage{array,arydshln}
\usepackage{geometry}
\usepackage{url}
\usepackage{ulem}
\theoremstyle{plain}
\newtheorem{thm}{Theorem}[section]
\newtheorem{pro}[thm]{Proposition}
\newtheorem{`thm'}[thm]{``Theorem''}
\newtheorem{cor}[thm]{Corollary}
\newtheorem{lem}[thm]{Lemma}

\newtheorem{dfn-thm}[thm]{Definition-Theorem}
\newtheorem{dfn-pro}[thm]{Definition-Proposition}
\newtheorem*{mainthm}{Main-Theorem}
\theoremstyle{definition}
\newtheorem{prob}[thm]{Problem}
\newtheorem{dfn}[thm]{Definition}

\newtheorem{asm}[thm]{Assumption}

\theoremstyle{remark}
\newtheorem{rmk}[thm]{Remark}
\newtheorem{rmks}[thm]{Remarks}

\newtheorem{exs}[thm]{Examples}

\newcommand{\bb}[1]{\mathbb{#1}}
\newcommand{\fra}[1]{\mathfrak{#1}}
\newcommand{\ca}[1]{\mathcal{#1}}
\newcommand{\wt}[1]{\widetilde{#1}}

\newcommand{\CL}[1]{\mathit{Cliff}(\bb{R}^{#1})}
\newcommand{\Cl}{\mathit{Cliff}}

\newcommand{\fib}{{\rm fib}}

\newcommand{\id}{{\rm id}}

\newcommand{\End}{{\rm End}}

\newcommand{\midd}{\,\middle|\,}

\newcommand{\rank}{{\rm rank}}

\newcommand{\red}{{\rm red}}

\newcommand{\Ad}{{\rm Ad}}

\newcommand{\bra}[1]{\left(#1\right)}
\newcommand{\bbra}[1]{\left\{#1\right\}}
\newcommand{\bbbra}[1]{\left[#1\right]}

\newcommand{\vect}[1]{\overrightarrow{#1}}

\newcommand{\ev}{{\rm ev}}
\newcommand{\odd}{{\rm odd}}
\newcommand{\od}{{\rm od}}

\newcommand{\vep}{\varepsilon}
\newcommand{\grotimes}{\widehat{\otimes}}
\newcommand{\groplus}{\widehat{\oplus}}

\begin{document}
\title{Geometric $K$-homology and operator $K$-theory for Hilbert manifolds}
\author{Doman Takata \\
Niigata University}

\date{\today}

\maketitle
\begin{abstract}
Poincar\'e duality is a classical theorem relating the homology and cohomology of closed oriented manifolds.
This theorem has been extended to more general settings and to generalized (co)homology theories, including $K$-theory.
In this paper, we construct an infinite-dimensional analogue of the $K$-theoretic Poincar\'e duality homomorphism.
More precisely, for an infinite-dimensional Hilbert manifold $\mathcal{M}$, we construct a homomorphism 
$$K^{geo}_*(\mathcal{M})\to K_*(\mathcal{A}(\mathcal{M})),$$
where $K^{geo}_*(\mathcal{M})$ denotes the geometric $K$-homology of Baum and Douglas, and $\mathcal{A}(\mathcal{M})$ is a $C^*$-algebra associated to $\mathcal{M}$, based on a construction of Gong, Wu, and Yu.
We also prove that the constructed homomorphism is non-trivial in certain cases.
\end{abstract}

\setcounter{tocdepth}{2}
\tableofcontents

\section{Introduction}

Poincar\'e duality is a fundamental theorem relating homology and cohomology. It has numerous applications, including intersection theory, and has been generalized in various directions \cite{Hat,Swi}.

In this paper, we construct an infinite-dimensional analogue of the $K$-theoretic Poincar\'e duality map.

\begin{mainthm}
Let $\ca{M}$ be a Hilbert manifold satisfying Assumption \ref{Katei},
let $\ca{A(M)}$ be the $C^*$-algebra associated to $\ca{M}$ introduced in Definition \ref{def A(M)},
and let $K_*^{geo}(\ca{M})$ denote the Baum--Douglas geometric $K$-homology of $\ca{M}$ (see Definition \ref{Def of K-homology}).
Then there exists a homomorphism
$$
\Phi:K_*^{geo}(\ca{M})\to K_*(\ca{A(M)}).
$$
\end{mainthm}

This theorem connects the topology of $\ca{M}$ to the noncommutative geometry of $\ca{A(M)}$.
This homomorphism provides a new approach to study $\ca{A(M)}$ through the topology of $\ca{M}$.
For instance, whenever $\Phi$ is nontrivial, one obtains nontrivial elements in the $K$-theory of $\ca{A(M)}$ from topological data of $\ca{M}$.
Indeed, we prove that $\Phi$ is nontrivial for certain Hilbert manifolds.

We next review the background and previous work related to our construction.

\subsection*{Homological Poincar\'e duality}

We first review a description of Poincar\'e duality entirely in terms of Gysin maps. This viewpoint will be essential in the proof of the main theorem.
For a treatment of Poincar\'e duality from the 
viewpoint of algebraic topology, see \cite{Hat,Swi}.

The simplest form of Poincar\'e duality is an isomorphism
$PD_X:H_*(X)\to H^{\dim (X)-*}(X)$ for a closed oriented manifold $X$.
Using Poincar\'e duality, we can define the Gysin map (also called the wrong-way map) on cohomology as follows. For a continuous map $f:M\to X$ of closed oriented manifolds, $f_!:=PD_X\circ f_*\circ PD_M^{-1}$, where $PD_M:H_*(M)\to H^{\dim (M)-*}(M)$ is the Poincar\'e duality map.
The Gysin map encodes intersection-theoretic information.
In fact, if $\iota_1:M_1\hookrightarrow X$ and $\iota_2:M_2\hookrightarrow X$ intersect transversally, and if $\dim(M_1)+\dim(M_2)=\dim (X)$, we have $\int_X(\iota_1)_!(1)\wedge (\iota_2)_!(1)=\#(M_1\cap M_2)$.
See \cite{Nic} for details.

We can also define the Gysin map without Poincar\'e duality.
A smooth map $f:M\to X$ of closed oriented manifolds admits the following factorization: the zero section of a vector bundle $s:M\to \nu$, where $\nu$ is the total space of the normal bundle of $(f,\id):M\hookrightarrow X\times M$, the open embedding $\iota:\nu \hookrightarrow M\times X$, and the projection $\varpi:X\times M\to X$.
Then we define $s_!:H^*(M)\to H_c^{*+\dim (X)}(\nu)$ by the Thom isomorphism, $\iota_!:H_c^{*+\dim (X)}(\nu)\to H^{*+\dim (X)}(X\times M)$ by extension by zero, and $\varpi_!:H^{*+\dim (X)}(X\times M)\to H^{*+\dim (X)-\dim(M)}(X)$ by the fiber integration.
Then, we define $f_!:=\varpi_!\circ \iota_!\circ s_!$.
Moreover, if $f$ is an embedding, $f_!(1)$ is nothing but the Poincar\'e dual of the submanifold $M$ as explained in \cite{BT}.

Combining this construction with Jakob's bordism-type description of homology theories \cite{Jak}, we can describe Poincar\'e duality.
According to \cite{Jak}, if $X$ has the homotopy type of a finite CW-complex, every element of $H_*(X)$ can be represented by a triple consisting of an oriented manifold $M$, a continuous map $f:M\to X$ and an element $u\in H^*(M)$.
$H_*(X)$ is isomorphic to the group of equivalence classes of such triples.
To such a triple, one can assign an element $f_!(u)\in H^{*+\dim(X)-\dim(M)}(X)$.
One can prove that $f_!(u)$ is independent of the choice of representatives.

It follows directly from the definitions that this construction coincides with the classical Poincar\'e duality.
In fact, the homology element corresponding to the triple $(M,u,f)$ is $f_*(PD_M^{-1}(u))$.
Its Poincar\'e dual is $PD_X(f_*(PD_M^{-1}(u)))=f_!(u)$.
Consequently, Poincar\'e duality admits a description entirely in terms of Gysin maps.

This description also extends naturally to non-orientable manifolds.
If $X$ is a closed manifold, not necessarily orientable, let $o_X$ denote its orientation sheaf.
Let $(M,u,f)$ be a triple representing an element of $H_*(X)$.
The normal bundle $\nu$ for $M\hookrightarrow X\times M$ can be non-orientable.
The Thom isomorphism remains valid for non-orientable vector bundles with coefficients in the orientation sheaf $o^{fib}_\nu$ of the vector bundle $\nu$, namely $H^*(M)\cong H_c^{*+\rank(\nu)}(\nu;o^{fib}_\nu)$.
Since $M$ is oriented, we have a specified isomorphism $H^*(M)\cong H^*(M;o_M)$.
Let $o_\nu$ be the orientation sheaf of the total space of $\nu$.
Using the canonical isomorphisms $o_M\otimes o^{fib}_{\nu}\cong o_\nu$ and $o_X|_{\nu}\cong o_\nu$, the Poincar\'e duality map takes the form $H_*(X)\to H^{\dim(X)-*}(X;o_X)$.

We can also consider non-compact manifolds.
In this case, the Poincaré duality map takes the form $H_*(X)\to H^{\dim(X)-*}_c(X)$.
It reflects the fact that the Thom isomorphism and the extension by zero take values in the compactly supported cohomology.
Combining these two generalizations, the Poincar\'e duality map takes the following form for non-compact non-orientable $X$ 
$$PD_X:H_*(X)\to H^{\dim(X)-*}_c(X;o_X).$$

\subsection*{$K$-theory Poincar\'e duality}

Poincar\'e duality is not restricted to singular homology and cohomology;
it also holds for generalized homology and cohomology theories \cite{Swi,Dye}.
In particular, for a $Spin^c$-manifold $X$, the $K$-homology of $X$ is isomorphic to the $K$-cohomology of $X$ by Poincar\'e duality.
One can describe this Poincar\'e dualityin exactly the same manner as above: by using the Gysin map defined directly, and the bordism-type description of the generalized homology.

In fact, the Gysin map in $K$-theory can be defined by the composition of the Thom isomorphism, extension by zero and fiber integration \cite{Fur,ASi1,CW,Kar}.
In \cite{BD}, Baum and Douglas introduced geometric $K$-homology of a topological space $X$, denoted by $K^{geo}_*(X)$, as the group of equivalence classes of triples $(M,E,f)$ consisting of a closed $Spin^c$-manifold $M$, a continuous map $f:M\to X$ and a complex vector bundle $E$ on $M$.
Then, the Poincar\'e dual of $(M,E,f)$ is given by $f_!([E])$.
One can easily check that $f_!([E])$ coincides with the classical Poincar\'e dual of the corresponding $K$-homology element.

The $K$-theoretic counterpart of the orientation sheaf $o_X$ on $X$ is the Clifford algebra bundle $\Cl(TX)$ of $TX$.
Let $Cl_\tau(X)=C_0(X,\Cl(TX))$.
The $K$-theory group $K_0(Cl_\tau(X))$ is the $K$-theoretic analogue of $H^{\dim(X)}(X;o_X)$.
In fact, if $X$ is parallelizable, $Cl_\tau(X)\cong C_0(X,\Cl(\bb{R}^{\dim(X)}))$ and we have
$K_0(Cl_\tau(X))\cong K_0(C_0(X)\grotimes \Cl(\bb{R}^{\dim(X)}))
\cong K_{\dim(X)}(C_0(X))$.
The appearance of the Clifford algebra absorbs the degree shift in the classical formulation of Poincar\'e duality.
Then, the Poincar\'e duality map takes the following form:
$$PD_X:K_*^{geo}(X)\to K_{*}(Cl_\tau(X)).$$
Therefore, the Poincar\'e duality map may be viewed as a bridge between covariant topological information of $X$ 
and covariant noncommutative geometric information of $Cl_\tau(X)$.

Besides the Poincar\'e duality described above, there is another formulation of $K$-theoretic Poincar\'e duality due to Kasparov \cite{Kas15}.
This duality takes the form
$KK(C_0(X),\bb{C})\to \ca{R}KK(X;C_0(X),Cl_\tau(X))$.
The Poincar\'e duality described above corresponds to the duality between compactly supported cohomology and homology, whereas Kasparov's formulation is analogous to the duality between ordinary cohomology and Borel--Moore homology.
Kasparov's Poincar\'e duality was generalized to an infinite-dimensional setting in \cite{T2,T3}.

\subsection*{Infinite-dimensional manifolds and noncommutative geometry}

The goal of this paper is to develop an infinite-dimensional version of the above $K$-theoretic Poincar\'e duality by replacing the finite-dimensional manifold $X$ with an infinite-dimensional Hilbert manifold $\mathcal M$.

For this purpose, we need an infinite-dimensional analogue of $Cl_\tau(X)$.
Since an infinite-dimensional manifold is nowhere locally compact, every continuous function vanishing at infinity is identically zero.
This prevents us from applying the finite-dimensional construction directly.
Moreover, by the Gelfand-Naimark theorem, the category of locally compact Hausdorff spaces and that of commutative $C^*$-algebras are contravariantly equivalent.
Thus, one needs to formulate an alternative notion playing the role of ``vanishing at infinity'', and assign a noncommutative $C^*$-algebra to an infinite-dimensional manifold.

This line of research originated in the work of Higson--Kasparov--Trout \cite{HKT} and Higson--Kasparov \cite{HK}.
They introduced a $C^*$-algebra $\ca{A}(H)$ for a Hilbert space $H$ in order to study an infinite-dimensional version of the Bott periodicity theorem.
The $C^*$-algebra $\ca{A}(H)$ may be regarded as an infinite-dimensional analogue of the graded suspension of $Cl_\tau(\bb{R}^n)$.
In fact, $\ca{A}(H)$ is the inductive limit of the graded suspensions of $Cl_\tau(\bb{R}^n)$.
This is made possible by constructing compatible homomorphisms between the graded suspensions of $Cl_\tau(\bb{R}^n)$ and $Cl_\tau(\bb{R}^{n+k})$.
Consequently, the Bott periodicity map is realized as the map induced by a $*$-homomorphism.

About two decades later, Gong--Wu--Yu gave another description of $\ca{A}(H)$, and generalized it to Hilbert--Hadamard spaces. 
These are (possibly infinite-dimensional) geodesic metric spaces of non-positive curvature in the sense of Alexandrov in \cite{GWY}.
See also \cite{GWXY}.
Subsequently, in a conference talk \cite{Yu}, Yu generalized this construction to Hilbert manifolds whose injectivity radius is bounded below, and the author wrote down the details and investigated the resulting construction in \cite{T2,T3}.
A different approach, based on the original inductive-limit construction of Higson--Kasparov--Trout, was developed in \cite{Tro,DT}.

Computing noncommutative geometric invariants of such $C^*$-algebras is an important problem.
For instance, the $K$-theory of $\ca{A}(H)$ for a Hilbert space $H$ was computed in \cite{HKT}.
If $\ca{M}$ is a Hilbert-Hadamard space admitting a sequence $M_1\subseteq M_2\subseteq \cdots\subseteq \ca{M}$ such that each $M_i$ is finite-dimensional and totally geodesic, the Bott map is shown to be injective in \cite{GWY}. 
Consequently, $K_*(\ca{A(M)})$ is non-trivial.
The Poincar\'e duality homomorphism constructed in this paper is expected to provide a useful tool for studying such invariants.

\subsection*{The outline of the construction}

We now outline the construction of an infinite-dimensional analogue of the $K$-theoretic Poincar\'e duality 
$PD_X:K_*^{geo}(X)\to K_*(Cl_\tau(X))$.
The left-hand side requires no modification, since geometric $K$-homology is defined for arbitrary topological spaces.
For the right-hand side, we use $\ca{A(M)}$ as a substitute for the   function algebra of $\ca{M}$. 
The necessary properties of $\ca{A(M)}$ are reviewed in Section \ref{subsection A(M)}.

The key ingredient in the construction is the Gysin map.
In the classical definition, it is just a homomorphism between $K$-theory groups.
In \cite{CS,EM}, Connes--Skandalis and Emerson--Meyer realized Gysin maps as Kasparov products with suitable $KK$-elements.
We follow this framework.
We will construct a $KK$-element implementing the Gysin map associated to a continuous map from a closed manifold to $\ca{M}$ in Section \ref{subsection PF for Hilb}.

Using these ingredients, we can assign an element of 
$K_*(\mathcal A(\mathcal M))$ to a geometric $K$-homology cycle for $\ca{M}$.
We then prove that this assignment depends only on the equivalence class of geometric $K$-homology cycles.
Therefore, there exists a group homomorphism $\Phi:K_*^{geo}(\ca{M})\to K_*(\ca{A(M)})$.

Finally, we show that the resulting homomorphism is non-trivial for certain Hilbert manifolds by computing the Kasparov product with a suitable $K$-homology class.
The resulting pairing admits an intersection-theoretic description.
This description supports the view that $\Phi$ is an infinite-dimensional analogue of the $K$-theoretic Poincar\'e duality.

\section{Background and conventions}

In this section, we summarize the basic material needed throughout the paper and fix the notation and conventions. Since most of the material is standard, the reader may safely skip this section and consult it only when necessary.
However, we recommend the reader to read Section \ref{Other notational remarks}, as it contains some notation and conventions that may not be standard.

\subsection{Clifford algebras}


Throughout this paper, 
we use the following definition of the Clifford algebra.

\begin{dfn}\label{Def of Cliff}
For a finite-dimensional Euclidean space $V$, we define the {\it complex Clifford algebra of $V$} by $\Cl(V):=T(V)\otimes \bb{C}/(v^2- \|v\|^2)$.
The involution $\epsilon=-\id_V$ uniquely extends to a ring automorphism on $\Cl(V)$ (e.g., $v$ is odd and $vw$ is even for $v,w\in V$), which defines the $\bb{Z}_2$-grading on $\Cl(V)$.
The assignment $v^*:= v$ uniquely extends to an conjugate linear  anti-automorphism on $\Cl(V)$ (e.g., $(vw)^*=wv$ for $v,w\in V$), which defines the $*$-algebra structure on $\Cl(V)$.
\end{dfn}

Since this construction is independent of the choice of basis, the following makes sense.

\begin{dfn}
$(1)$ For a Euclidean vector bundle $E$ on a topological space $M$, we define $\Cl(E)$ in the obvious way.

$(2)$ For a finite-dimensional Riemannian manifold $M$, 
the $C^*$-algebra $C_0(M,\Cl(TM))$ is denoted by $Cl_\tau(M)$.
\end{dfn}

\begin{rmk}
One can see that $\Cl(V)$ is a $C^*$-algebra as follows: introduce an inner product on $\Cl(V)$ such that all monomials $e_{i_1}e_{i_2}\cdots e_{i_k}$ are of unit length and mutually orthogonal for an orthonormal basis $\{e_1,\ldots,e_n\}$; consider the faithful representation given by left multiplication $\Cl(V)\to \End(\Cl(V))$; and equip $\Cl(V)$ the $C^*$-algebra structure induced from $\End(\Cl(V))$.
Hence, in the above setting, $C_0(M,\Cl(E))$ has a $C^*$-algebra structure in the obvious way.
\end{rmk}

We next fix our convention for graded tensor products.
For two $\bb{Z}_2$-graded algebras 
$A=A_0\groplus A_1$ and $B=B_0\groplus B_1$, the graded tensor product $A\grotimes B$ is defined as follows: the underlying vector space is just $A\otimes B$, and the multiplication is twisted by the grading: $(a_1\grotimes b_1)\cdot(a_2\grotimes b_2)
:=(-1)^{|b_1||a_2|}(a_1a_2)\grotimes (b_1b_2)$, where $|b|=0$ for $b\in B_0$ and $|b|=1$ for $b\in B_1$, and similarly for $|a|$, $a\in A_0\cup A_1$. 
Intuitively, exchanging two odd elements introduces a minus sign.
For example, we have
$\Cl(V\oplus W)\cong 
\Cl(V)\grotimes
\Cl(W)$ for the orthogonal direct sum $V\oplus W$.
We also use this convention for tensor products of operators:
$(F\grotimes G)\cdot(a\grotimes b)
:=(-1)^{|a||G|}F(a)\grotimes G(b)$.


A $\bb{Z}_2$-graded $*$-representation of $\Cl(\bb{R}^m)$ on a $\bb{Z}_2$-graded Hermitian vector space $S$ is a $*$-homomorphism $\gamma:\Cl(\bb{R}^m)\to \End(S)$ preserving the $\bb{Z}_2$-grading.
Such an $S$ is called a {\bf Clifford module}.
When we regard $S$ as a module over $\Cl(\bb{R}^m)$, the left action of $\Cl(\bb{R}^m)$ on $S$ is called a {\bf Clifford multiplication}.
A Clifford module is said to be irreducible if it has no nontrivial invariant subspace.
An irreducible module over $\Cl(V)$ is also called a {\bf spinor module of $V$}.

It is known that
$\Cl(\bb{R}^2)$ has a unique irreducible representation up to isomorphism, given as follows. 
For the standard basis $\{e_1,e_2\}$, we define
$$
\gamma(e_1)=
\begin{pmatrix}
0 & 1 \\
1 & 0
\end{pmatrix},\ \ \ 
\gamma(e_2)=
\begin{pmatrix}
0 & \sqrt{-1} \\
-\sqrt{-1} & 0
\end{pmatrix}.$$
Then, $\gamma$ uniquely extends to a $*$-representation of $\Cl(\bb{R}^2)$.
There are two choices of $\bb{Z}_2$-grading structure.
We regard the grading automorphism $\epsilon=\sqrt{-1}\gamma(e_1e_2)$ as the preferred one ($-\sqrt{-1}e_1e_2$ is the other choice of the grading automorphism).
We denote this representation by $\Delta^2$.
Using the isomorphism $\Cl(\bb{R}^{2n})\cong \Cl(\bb{R}^{2})^{\grotimes n}$, we can construct the spinor module of $\bb{R}^{2n}$ by $(\Delta^2)^{\grotimes n}$.

One can check that the construction of the previous paragraph depends only on the orientation of $\bb{R}^2$ and $\bb{R}^{2n}$.
The same argument works for an oriented Euclidean vector space $V$ of dimension $2n$
: For an oriented orthonormal basis $\{e_1,e_2,\ldots,e_{2n}\}$, the involution $\epsilon=\sqrt{-1}^n\gamma(e_1e_2\cdots e_{2n})$ is the preferred grading automorphism.

\medskip


The following notion is equivalent to a Clifford module structure. It will be used to describe $Spin^c$-structures.

\begin{dfn}
An $n$-multigrading on a $\bb{Z}_2$-graded linear space $M$ consists of $n$ involutions $\epsilon_1$, $\epsilon_2$, $\cdots$, $\epsilon_n$ reversing the grading, and satisfying 
$(\epsilon_i)^2=\id$ and $\epsilon_i\epsilon_j=-\epsilon_j\epsilon_i$ for $i\neq j$.
When $M$ is a right module over an algebra $A$, we require 
each $\epsilon_i$ to be an $A$-module homomorphism, namely $(\epsilon_i m)a=\epsilon_i (ma)$ for $m\in M$ and $a\in A$.
\end{dfn}

In fact, a Clifford module structure $\gamma:\Cl(\bb{R}^n)\to \End(S)$ gives an $n$-multigrading structure:
$\gamma(e_i)$ ($i=1,2,\cdots n$) are the required involutions.
Conversely, an $n$-multigraded $\bb{Z}_2$-graded linear space $M$ admits the structure of a Clifford module over $\bb{R}^n$ by $\gamma(e_i):=\epsilon_i$.

We conclude this subsection with a comparison of the above definition and the standard one.

\begin{rmks}
$(1)$ For a real vector space $V$ equipped with a quadratic form $Q$, the Clifford algebra $\Cl(V,Q)$ is defined by $T(V)\otimes \bb{C}/(v^2+Q(v)1)$.
In this notation, our Clifford algebra is the case when $Q(v)=-\|v\|^2$.

$(2)$ $\Cl(V,\|\bullet\|^2)=T(V)\otimes \bb{C}/(v^2+\|v\|^2)$ is the most standard complex Clifford algebra.
For the moment, we denote it by $\Cl_+(V)$.
This Clifford algebra is used in \cite{BHS} to define $n$-multigradings.

$(3)$ Even if we use $\Cl(V)$, $\Cl_+(V)$ naturally appears when considering the graded right multiplication:
We define $R:V\to \End(\Cl(V))^{op}$ by
$$R(v)(w):=\epsilon(w)v$$
for $v\in V$ and $w\in \Cl(V)$.
Then
$$R(v)^2(w)=R(v)(\epsilon(w)v)
=R(v)\bbra{(-1)^{|w|}wv}
=(-1)^{|wv|}(-1)^{|w|}wv^2
=(-1)^{|w|+1}(-1)^{|w|}wv^2
=-w\|v\|^2.$$
Therefore, $R(v)^2=-\|v\|^2$, and $R$ factors through $\Cl_+(V)$.

$(4)$ The standard one $\Cl_+(V)$ is isomorphic to $\Cl(V)$ by $v\mapsto \sqrt{-1}v$.
Thus, throughout the remainder of this paper, we always use the Clifford algebra defined in Definition \ref{Def of Cliff}.
\end{rmks}

\subsection{$KK$-theory}

We refer the reader to \cite{Bla} for a standard account of $KK$-theory.


The $KK$-group $KK(A,B)$ is an abelian group associated to a pair of separable $C^*$-algebras $(A,B)$.
It simultaneously generalizes $K$-homology of $A$ and $K$-theory of $B$.

\begin{exs}
$(1)$ For a compact Hausdorff space $B$ and a complex vector bundle $E$, there is a $KK$-element $[E]\in KK(\bb{C},C(B))$.

$(2)$ For a Riemannian manifold $M$, we define $\gamma^*:TM\to \End(\Cl(TM))$ by the graded right multiplication $\gamma^*(v)w:=(-1)^{|w|}wv$.
As in \cite{Kas88}, the Dirac element of $M$ is defined by
$$[d_M]:=\bra{L^2(M,\Cl(TM)),\pi,\frac{D}{\sqrt{1+D^2}}}\in KK(Cl_\tau(M),\bb{C}),$$
where $\pi$ is the left multiplication of $Cl_\tau(M)$,
 and $D$ is
$$D=\sum_j\gamma^*(\partial_j)\partial_j$$
for an orthonormal frame $\{\partial_j\}$ of $TM$.

$(3)$ Let $A$ and $B$ be separable $C^*$-algebras, and let $\phi:A\to B$ be a $*$-homomorphism.
Then $\phi$ determines a $KK$-element $[\phi]$ by
$(B,\phi,0)\in KK(A,B)$.
In particular, a continuous proper map $f:X\to Y$ between locally compact Hausdorff spaces gives a $*$-homomorphism $f^*:C_0(Y)\to C_0(X)$, and hence a $KK$-element $[f^*]\in KK(C_0(Y),C_0(X))$.
In contrast, an open embedding $\iota:U\to X$ into a locally compact topological space gives a $*$-homomorphism $\iota_*:C_0(U)\to C_0(X)$ by the extension by zero, and hence a $KK$-element $[\iota_*]\in KK(C_0(U),C_0(X))$.
\end{exs}

By abuse of notation, we sometimes write an unbounded Kasparov module \cite{BJ,Kuc} in place of its associated bounded Kasparov module.
An unbounded Kasparov module is transformed into a bounded one by
$$(E,\pi,D)\mapsto \bra{E,\pi,\frac{D}{\sqrt{1+D^2}}}.$$


For $C^*$-algebras $A$, $B$ and $C$, assuming that $A$ and $B$ are separable and $C$ is $\sigma$-unital, the Kasparov product 
$KK(A,B)\times KK(B,C)\to KK(A,C)$ is defined.
The Kasparov product of $u\in KK(A,B)$ and $v\in KK(B,C)$ is denoted by $u\grotimes_{B}v$.
We often omit the subscript $B$.
When $B=\bb{C}$, the Kasparov product is called the external Kasparov product.
This product is associative, and the identity element $[\bm{1}_A]$ is given by $(A,\id,0)$.
The Kasparov product extends to
$$KK(A_1,B\grotimes C_1)\times
KK(B\grotimes A_2,C_2)
\to KK(A_1\grotimes A_2,C_1\grotimes C_2).$$
We use the same notation $u\grotimes v$ for this generalized product.

Associativity allows us to define the category $\fra{KK}$. Its objects are separable $C^*$-algebras, and morphisms from $A$ to $B$ are elements of $KK(A,B)$. See \cite{Mey} for details.
We use this category only to simplify notation.
For example, an equality of two Kasparov products can be rephrased as a commutative diagram in the category $\fra{KK}$.

\medskip

We define $KK^n(A,B):=KK(\CL{n}\grotimes A,B)$ following \cite{Kas88}.
Equivalently, one can define it using $n$-multigraded Kasparov $(A,B)$-modules.

Since $\Cl(\bb{R}^2)$ is $KK$-equivalent to $\bb{C}$,
we have $KK^n\cong KK^{n+2}$.
Following \cite{BHS}, we call this the {\it formal Bott periodicity}.
This isomorphism is given by the external Kasparov product with the canonical spinor module for $\Cl(\bb R^2)$:
$u\mapsto [\Delta^2]\grotimes_{\bb{C}}u$.
Using these canonical isomorphisms, we define
$$KK^\ev(A,B):=\varinjlim_{n}KK^{2n}(A,B),\ \ \ 
KK^\odd(A,B):=\varinjlim_{n}KK^{2n+1}(A,B).$$
By the formal Bott periodicity, 
$KK^n(A,B)$ is naturally isomorphic to $KK^\ev(A,B)$ if $n$ is even, and to $KK^\odd(A,B)$ if $n$ is odd.

\subsection{$Spin^c$-structure}

We adopt the following definition of a $Spin^c$-structure, following \cite{BHS}.

\begin{dfn}
For an $m$-dimensional smooth manifold $M$, a $Spin^c$-structure is a pair consisting of a Riemannian metric on $M$ and an $m$-multigraded Hermitian vector bundle $S_M$ of rank $2^m$ equipped with a {\it right} action of $\Cl(TM)$ commuting with the multigrading automorphisms. Such a vector bundle is called an {\bf $m$-multigraded spinor bundle}.
\end{dfn}
\begin{rmk}
If one has a Hermitian vector bundle $S$ equipped with a {\it left} action of $\Cl_+(TM)$, one can define a {\it right} action $\gamma$ of $\Cl(TM)$ by the graded left multiplication $s\cdot v:=(-1)^{|s|}\gamma(v)s$ for $s\in S$ and $v\in TM$.
\end{rmk}

Using this notion, we can compare $C_0(M)$ and $Cl_\tau(M)$ in the Kasparov category.

\begin{lem}\label{KK-equivalence C0 Cl}
For an $m$-dimensional $Spin^c$-manifold $M$, $C_0(M)\grotimes \Cl(\bb{R}^m)$ and $Cl_\tau(M)$ are $KK$-equivalent.
\end{lem}

We only construct an element of $KK(C_0(M)\grotimes \Cl(\bb{R}^m),Cl_\tau(M))$, which is used later.

Let $Spin^c(n)$ be the subgroup of $\Cl(\bb{R}^n)^\times$ generated by the products of evenly many vectors of unit length and complex numbers of modulus one.
Let $R:Spin^c(m)\to U(\Cl(\bb{R}^m))$ be the representation given by the right multiplication: $R(g)w=wg^{-1}=wg^*$, and let 
$\Ad:Spin^c(m)\to U(\Cl(\bb{R}^m))$ be the representation given by the adjoint action: $\Ad(g)w=gwg^{-1}=gwg^*$.
This action preserves $\bb{R}^m$ and it gives a representation
$\Ad:Spin^c(m)\to O(\bb{R}^m)$.
Then, for a $Spin^c$-manifold $M$ equipped with a spinor bundle $S_M$, there exists a principal $Spin^c(m)$-bundle $P$ such that 
$TM\cong P\times_{Spin^c(m),\Ad}\bb{R}^m$, and $S_M\cong P\times_{Spin^c(m),R}\Cl(\bb{R}^m)$.
Observe that $S_M$ admits a left $\Cl(\bb{R}^m)$-module structure by the left multiplication.
Note that $\Cl(TM)\cong P\times_{Spin^c(m),\Ad}\Cl(\bb{R}^m)$.
Then, the formula
$[p,w]\cdot[p,v]:=[p,wv]$
gives a right $\Cl(TM)$-module structure
for $[p,w]\in S_M=P\times_{Spin^c(m),R}\Cl(\bb{R}^m)$ and
$[p,v]\in \Cl(TM)=P\times_{Spin^c(m),\Ad}\Cl(\bb{R}^m)$,
and the formula 
$\bra{[p,w_1],[p,w_2]}{}:=[p,(w_1)^*w_2]$
gives a $\Cl(TM)$-valued inner product for $[p,w_1],[p,w_2]\in S_M=P\times_{Spin^c(m),R}\Cl(\bb{R}^m)$.
The left action $\pi\grotimes \gamma$ of $C_0(M)\grotimes \Cl(\bb{R}^m)$ is given by the tensor product of the pointwise scalar multiplication $\pi$ and the multigrading structure $\gamma$.
Then, 
$$[S_M]:=(C_0(M,S_M),\pi\grotimes \gamma,0)\in KK(C_0(M)\grotimes \Cl(\bb{R}^m),Cl_\tau(M))$$
gives a $KK$-equivalence.
We leave the construction of the inverse element to the reader.

When $M$ is even-dimensional, the following makes sense.

\begin{dfn}
Let $M$ be an $m$-dimensional Riemannian manifold.
Suppose that $m$ is even.
Then, a {\bf reduced spinor bundle} is a $\bb{Z}_2$-graded Hermitian bundle $S$ of rank $2^{\frac{m}{2}}$ equipped with a {\it right} action of $\Cl(TM)$ commuting with the grading automorphism.
\end{dfn}
\begin{rmks}
$(1)$ Spinor bundle and reduced spinor bundle with the same $Spin^c$-structure correspond to each other by the formal Bott periodicity.

$(2)$ By a similar construction of Lemma \ref{KK-equivalence C0 Cl}, for an even dimensional $Spin^c$-manifold $M$, $C_0(M)$ and $Cl_\tau(M)$ are $KK$-equivalent.
\end{rmks}

We define the restriction of the $Spin^c$-structure to the boundary.

\begin{dfn-pro}[{\cite[Definition 3.5]{BHS}}]
For an $n$-dimensional manifold with boundary $L$ equipped with an $n$-multigraded spinor bundle $S_L$, we define the restriction of the $Spin^c$-structure to $\partial L$ as follows.
Let $\nu$ be the outward pointing unit normal vector field on $\partial L$. Then the formula
$$X(u):=-(-1)^{|u|}\epsilon_1u \nu$$
defines an even self-adjoint involution.
Then, the $(+1)$-eigenspace of $X$ is an $(n-1)$-multigraded spinor bundle of $\partial L$.
\end{dfn-pro}

\begin{rmk}
The above definition differs from that of {\cite[Definition 3.5]{BHS}}, because we define the Clifford algebra 
using the negative-definite inner product.
\end{rmk}

\subsection{Other notational remarks}\label{Other notational remarks}

We adopt the following notational conventions.

\begin{itemize}
\setlength{\parskip}{0cm} 
  \setlength{\itemsep}{0cm} 

\item For a $\bb{Z}_2$-graded vector space $V=V_0\widehat{\oplus}V_1$, we denote the grading by $|\bullet|$ and the grading automorphism by $\epsilon$, i.e., for $v\in V_i$ ($i=0$ or  $1$), $|v|=i$ and $\epsilon (v)=(-1)^iv=
(-1)^{|v|}v$.
An element of $V$ belonging to either $V_0$ or $V_1$ is said to be homogeneous.

\item  Unless otherwise stated, we denote projections associated to fiber bundles by $\varpi$. We often use the same symbol for different fiber bundles.

\item Almost all inclusions between spaces are denoted by $\iota$.
We often use the same symbol to denote different inclusions.

\item We reserve the symbol $\vep$ for the positive constant associated to a Hilbert manifold satisfying Assumption \ref{Katei}. Small positive numbers are denoted by $\delta$, $\eta$, and so on.

\item We write $\gamma$ for a Clifford multiplication.

\item Unless otherwise stated, we denote the $*$-homomorphism in a Kasparov module by $\pi$.
\end{itemize}

\section{Several essential $C^*$-algebras}

\subsection{The $C^*$-algebra $\ca{S}_\vep$ and the Bott periodicity}

The main object of study in this paper is the $C^*$-algebra $\ca{A(M)}$.
In order to define it, we need the $\bb{Z}_2$-graded $C^*$-algebra $\ca{S}_\vep$ defined as follows.

\begin{dfn}
For a fixed positive number $\vep$, we define $\ca{S}_\vep$ to be $C_0(-\vep,\vep)$.
The $\bb{Z}_2$-grading automorphism is given by $\epsilon(f)(t)=f(-t)$.
Thus, the even (resp. odd) part consists of the even (resp. odd) functions.
\end{dfn}

The $C^*$-algebra $\ca{S}_\vep$ plays an important role in this paper.
This is essentially because it has the multiplier $X:f(x)\mapsto xf(x)$.
This multiplier induces a bi-algebra structure on $\ca{S}_\vep$.
The following $*$-homomorphism $\Delta$ will play an important role in the computations in Section \ref{Our construction is non-trivial}.

\begin{lem}
The $*$-homomorphism
$$\Delta f:=f(X\grotimes 1+1\grotimes X)$$
gives a coassociative coproduct on $\ca{S}_\vep$.
\end{lem}
\begin{proof}
We divide the proof into several two steps.
See also \cite{HG}.

$(1)$ $\Delta(f)\in \ca{S}_\vep\grotimes \ca{S}_\vep$.
Indeed, we can represent $f$ by $f(x)=f_{\ev}(x^2)+xf_{\od}(x^2)$, where $f_{\ev}$ and $f_\od$ are suitable even functions.
Then, since $(X\grotimes 1+1\grotimes X)^2=X^2\grotimes 1+1\grotimes X^2$,
\begin{align*}
\Delta(f)(x,y)
&=\{f_{\ev}((X\grotimes 1+1\grotimes X)^2)+(X\grotimes 1+1\grotimes X)f_{\od}((X\grotimes 1+1\grotimes X)^2)\}(x,y) \\
&=f_{\ev}(x^2+y^2)+(x+y)f_{\od}(x^2+y^2).
\end{align*}
The right-hand side vanishes outside $(-\vep,\vep)\times(-\vep,\vep)$, and belongs to $\ca{S}_\vep\grotimes \ca{S}_\vep$.

$(2)$ $\Delta$ is coassociative. Indeed, we compute $(1\grotimes \Delta)\Delta(f)(x,y,z)$ and see that it coincides with $(\Delta\grotimes 1)\Delta(f)(x,y,z)$.
In the following, the multiplier $X$ for the $i$-th component is denoted by $X_i$ for $i=1,2,3$.
\begin{align*}
(1\grotimes \Delta)\Delta(f) 
&=
(1\grotimes \Delta)
\bbra{f_{\ev}((X_1\grotimes 1+1\grotimes X_2)^2)+
(X_1\grotimes 1+1\grotimes X_2)f_{\od}((X_1\grotimes 1+1\grotimes X_2)^2)} \\
&=
(1\grotimes \Delta)
\bbra{f_{\ev}(X_1^2\grotimes 1+1\grotimes X_2^2)+
(X_1\grotimes 1+1\grotimes X_2)f_{\od}(X_1^2\grotimes 1+1\grotimes X_2^2)} \\
&=f_{\ev}(X_1^2\grotimes 1+1\grotimes (X_2\grotimes 1+1\grotimes X_3)^2) \\
&\ \ \ 
+
(X_1\grotimes 1+1\grotimes (X_2\grotimes 1+1\grotimes X_3))f_{\od}(X_1^2\grotimes 1+1\grotimes (X_2\grotimes 1+1\grotimes X_3)^2) \\
&=f_{\ev}(X_1^2\grotimes 1+1\grotimes X_2^2\grotimes 1+1\grotimes X_3^2) \\
&\ \ \ 
+
(X_1\grotimes 1\grotimes 1+1\grotimes X_2\grotimes 1+1\grotimes 1\grotimes X_3)
f_{\od}(X_1^2\grotimes 1\grotimes 1+1\grotimes X_2^2\grotimes 1+1\grotimes1\grotimes X_3^2).
\end{align*}
Therefore, 
$(1\grotimes \Delta)\Delta(f) 
(x,y,z)=f_{\ev}(x^2+y^2+z^2)+(x+y+z)f_{\od}(x^2+y^2+z^2)$.
By the same argument, we obtain $(\Delta\grotimes 1)\Delta(f)(x,y,z)=f_{\ev}(x^2+y^2+z^2)+(x+y+z)f_{\od}(x^2+y^2+z^2)$.
\end{proof}

We introduce several $KK$-elements of $\ca{S}_\vep$.

\begin{dfn}
We define
$[b_{\pm}]\in KK(\bb{C},\ca{S}_\vep)$ by
$$[b_+]=\bra{\ca{S}_\vep,1,\frac{X}{\vep}},\text{ and }
[b_-]=\bra{\ca{S}_\vep,1,-\frac{X}{\vep}}.$$

Let $\ev_0:\ca{S}_\vep\to \bb{C}$ denote the evaluation homomorphism at $0$: $\ev_0(f):=f(0)$. It defines a $KK$-element $[\ev_0]\in KK(\ca{S}_\vep,\bb{C})$.
\end{dfn}

It is easy to see 
$[b_+]\grotimes_{\ca{S}_\vep}[\ev_0]=1\in KK(\bb{C},\bb{C})$.
We compute the Kasparov products of $[b_{\pm}]$ and $[\Delta]$.

\begin{pro}\label{b times delta}
In $KK(\bb{C},\ca{S}_\vep\grotimes \ca{S}_\vep)$, we have
$[b_+]\grotimes[\Delta]=[b_+]\grotimes_{\bb{C}}[b_+]$ and
$[b_-]\grotimes[\Delta]=[b_-]\grotimes_{\bb{C}}[b_-]$.
\end{pro}
\begin{proof}
We prove only the former one.

Let $\rho:[0,\infty)\to \bb{R}$ be a function satisfying
$$\rho(t)=\begin{cases}
1 & (t\leq \vep^2) \\
\frac{\vep}{\sqrt{t}}& (t\geq \vep^2).
\end{cases}$$
Let $F:=\rho(X_1^2\grotimes 1+1\grotimes X_2^2)\bra{\frac{X_1}{\vep}\grotimes 1+1\grotimes \frac{X_2}{\vep}}$.
Then, 
$$[F]:=\bra{\ca{S}_\vep\grotimes \ca{S}_\vep,1,F}$$
gives a Kasparov $(\bb{C},\ca{S}_\vep\grotimes \ca{S}_\vep)$-module.
In fact, the factor $\rho(X_1^2\grotimes 1+1\grotimes X_2^2)$ ensures that $F^2=1$ on the boundary of $[-\vep,\vep]\times[-\vep,\vep]$.
We show that both
$[\Delta]\grotimes[b_+]$ and $[b_+]\grotimes_{\bb{C}}[b_+]$ are represented by $[F]$.

For $[b_+]\grotimes_{\bb{C}}[b_+]$, it suffices to check that $F$ satisfies the condition to be a Kasparov product, and it is clear.

Let $U_\vep(0)$ be the $\vep$-neighborhood of the origin in $\bb{R}^2$.
Let $C_0^{ev}(U_\vep(0))$ be the set of continuous even functions on $U_\vep(0)$ vanishing at infinitiy.
Let us consider $[b_+]\grotimes[\Delta]
=\bra{\ca{S}_\vep,1,\frac{X}{\vep}}\grotimes \bra{\ca{S}_\vep\grotimes \ca{S}_\vep,\Delta,0}$.
The module is given by $\ca{S}_\vep\grotimes\bra{\ca{S}_\vep\grotimes\ca{S}_\vep}
\cong C_0^{ev}(U_\vep(0))\ca{S}_\vep\grotimes\ca{S}_\vep
$, where the isomorphism is given by
$g_1\grotimes (g_2\grotimes g_3)=g_1(X_1\grotimes 1+1\grotimes X_2)g_2\grotimes g_3$.
Note that the multiplier $X\grotimes 1$ acts as $X_1\grotimes 1+1\grotimes X_2$. In fact,
\begin{align*}
(Xg_1)\grotimes (g_2\grotimes g_3)
&=\bbra{(X_1\grotimes 1+1\grotimes X_2)g_1(X_1\grotimes 1+1\grotimes X_2)}g_2\grotimes g_3 \\
&=X_1\grotimes 1\cdot g_1(X_1\grotimes 1+1\grotimes X_2)g_2\grotimes g_3
+1\grotimes X_2\cdot g_1(X_1\grotimes 1+1\grotimes X_2)g_2\grotimes g_3.
\end{align*}
Let $F'=\frac{X_1}{\vep}\grotimes 1+1\grotimes \frac{X_2}{\vep}$.
Then, $[F']=\bra{C_0(U_\vep(0))\ca{S}_\vep\grotimes\ca{S}_\vep,1,\frac{X_1}{\vep}\grotimes 1+1\grotimes \frac{X_2}{\vep}}$ is a representative of the Kasparov product.

Let $U_\vep^t:=\bbra{(x,y)\in \bb{R}^2\midd x^2+y^2< \vep^2 \text{ or }{\rm Max}\{|x|,|y|\}< \vep t}$.
Let $E$ be the $\bb{C}$-$C([0,1])\grotimes {\ca{S}_\vep\grotimes\ca{S}_\vep}$-bimodule 
$$E:=C\bra{[0,1],\bbra{C_0(U_\vep^t)\ca{S}_\vep\grotimes\ca{S}_\vep}_{t\in [0,1]}}.$$
Note that 
$(\ev_0)_*E=C_0(U_\vep(0))\ca{S}_\vep\grotimes\ca{S}_\vep$ and
$(\ev_1)_*E=\ca{S}_\vep\grotimes\ca{S}_\vep$.
Since $\rho =1$ on $U_\vep(0)$, the restriction of $F$ to $C_0(U_\vep(0))\ca{S}_\vep\grotimes\ca{S}_\vep$ coincides with $F'$.
Then, $(E,1,\id\grotimes F)$ gives a homotopy between $[F]$ and 
$[F']$.
\end{proof}

For a Euclidean vector space $V$, we define the Clifford operator $C_V$ on $Cl_\tau(V)$ by
$$C_V(\phi)(v)=\gamma(v) \phi(v),$$
where $\gamma(v)$ is the Clifford multiplication by $v$ under the natural identification $T_vV\cong V$.
Then, the $KK$-element
$$[b_V]:=(Cl_\tau(V),1,C_V)\in KK(\bb{C},Cl_\tau(V))$$
is invertible and its inverse is $[d_V]$.
Therefore, $Cl_\tau(V)$ and $\bb{C}$ are $KK$-equivalent, and it is a $KK$-theoretic formulation of the Bott periodicity.

The $C^*$-algebra $\ca{S}_\vep$ enables us to formulate 
the Bott periodicity as a $*$-homomorphism, while $\Delta$ makes this formulation, in some sense, more functorial.

\begin{pro}
We consider the $*$-homomorphism $\beta_V:\ca{S}_\vep\to \ca{A}(V)$ by
$$\beta_V(g):=g(X\grotimes 1+1\grotimes C_V).$$

$(1)$ We have $[b_+]\grotimes [\beta_V]\grotimes[\ev_0]=[b_V]$.

$(2)$ For a Euclidean vector space $W$, under the identification $Cl_\tau(W)\grotimes Cl_\tau(V)\cong Cl_\tau(W\oplus V)$, we have 
$$\beta_{W\oplus V}=
(\beta_W\grotimes \id_{Cl_\tau(V)})\circ \beta_V:\ca{S}_\vep\to 
\ca{S}_\vep\grotimes Cl_\tau(V)\to
\ca{S}_\vep\grotimes Cl_\tau(W\oplus V).$$
\end{pro}
\begin{proof}
$(1)$ Let $\beta_V^0(g):=g(C)$.
Then, the composition $\ev_0\circ \beta_V$ is given by $\beta_V^0$.
Thus, $[\beta_V]\grotimes[\ev_0]$ is represented by
$(Cl_\tau(V),\beta_V^0,0)$.

The Hilbert module for $[b_+]\grotimes [\beta_V]\grotimes[\ev_0]$ is by definition $\ca{S}_\vep\grotimes_{\ca{S}_\vep,\beta_V}Cl_\tau(V)$.
We can regard it as a submodule of $Cl_\tau(V)$ by the identification
$g\grotimes \phi\mapsto g(C_V)\phi$ 
for $g\in \ca{S}_\vep$ and $\phi\in Cl_\tau(V)$.
With this identification, the operator $X\grotimes 1_{Cl_\tau(V)}$ on $\ca{S}_\vep\grotimes_{\ca{S}_\vep,\beta_V^0} Cl_\tau(V)$ corresponds to
$C_V$ by the same argument of Proposition \ref{b times delta}.
Therefore, the Kasparov product is represented by
$$\bra{C_0(U_\vep)Cl_\tau(V),1,C_V},$$
where $U_\vep$ is the $\vep$-neighborhood of the origin in $V$.
It is homotopic to the Bott element by the same argument of Proposition \ref{b times delta}.

$(2)$ It is easy to see $\beta_V=(\id\grotimes \beta_V^0)\circ \Delta$.
Therefore, by coassociativity of $\Delta$, we have
\begin{align*}
(\beta_W\grotimes \id_{Cl_\tau(V)})\circ \beta_V
&=(\bbra{(\id\grotimes \beta_W^0)\circ \Delta}\grotimes \id_{Cl_\tau(V)})\circ (\id\grotimes \beta_V^0)\circ \Delta \\
&=(\id\grotimes \beta_W^0\grotimes \id)\circ (\Delta\grotimes \id_{Cl_\tau(V)})\circ (\id\grotimes \beta_V^0)\circ \Delta \\
&=(\id\grotimes \beta_W^0\grotimes \id)\circ (\id\grotimes\id\grotimes \beta_V^0)\circ (\Delta\grotimes \id_{\ca{S}_\vep})\circ \Delta \\
&=(\id\grotimes \beta_W^0\grotimes \beta_V^0)
\circ( \id_{\ca{S}_\vep}\grotimes\Delta)\circ \Delta.
\end{align*}
Then, by
$$\beta_W^0\grotimes \beta_V^0(\Delta(g))
=g(C_W\grotimes 1+1\grotimes C_V)
=\beta_{W\oplus V}^0(g),$$
we have
$$(\id\grotimes \beta_W^0\grotimes \beta_V^0)\circ
(\id_{\ca{S}_\vep}\grotimes\Delta)\circ \Delta
=(\id\grotimes \beta_{W\oplus V}^0)\circ\Delta
=\beta_{W\oplus V}.$$
\end{proof}

See \cite{HKT} for the proof of the corresponding result for $\ca{S}=C_0(\bb{R})$ equipped with the $\bb{Z}_2$-grading given by the same formula as that for our $\ca{S}_\vep$
Their proof does not use the coproduct.

\subsection{The $C^*$-algebra $\ca{A(M)}$}\label{subsection A(M)}

We briefly review the construction of the $C^*$-algebra $\ca{A(M)}$  associated to a Hilbert manifold $\ca{M}$, following the exposition in \cite{T2,T3}. The original construction is due to \cite{Yu}.
See also \cite{GWY,GWXY}.
It is a generalization of \cite{HKT,HK}. Another generalization are studied in \cite{Tro,DT}. 

We impose the following condition on a Hilbert manifold $\ca{M}$.
For the definition of Hilbert manifolds, we refer to the reader to  \cite{Lan}.

\begin{dfn}[{\cite[Definition 5.1]{GWY}}]\label{def field of Clifford algebras}
We consider the space
$$\Pi'(\ca{M}):=\prod_{(x,t)\in \ca{M}\times [0,\vep)}\Cl(t\bb{R}\oplus T_x\ca{M}),$$
where 
$$t\bb{R}:=\begin{cases}
\bb{R} & (t\neq 0) \\
0 & (t=0).\end{cases}
$$
Then we consider a $C^*$-algebra
$$\Pi'_b(\ca{M}):=\bbra{s\in \prod(\ca{M})\ \middle|\  \|s(x,t)\|\text{ is bounded.}}$$
equipped with the pointwise algebraic operations (addition, multiplication and the adjoint) and the uniform norm.
\end{dfn}
\begin{rmk}
It is possible to replace $\Pi'(\ca{M})$ with 
$$\Pi(\ca{M}):=\prod_{x\in \ca{M}}\ca{S}_\vep\grotimes \Cl(T_x\ca{M})$$ 
because we have an isomorphism $\ca{S}_\vep\cong C_0([0,\vep),\Cl(t\bb{R}))$ given by $f_\ev+Xf_\od\mapsto [t\mapsto f_\ev(t)+f_\od(t)e_1]$, where $e_1$ is the standard basis vector of $t\bb{R}$ (note that $f_\od(0)=0$).
In what follows, we use this description.
\end{rmk}

The following definition is modeled on the $C^*$-algebra associated to a Hilbert-Hadamard space \cite[Definition 5.14]{GWY}.

\begin{dfn}\label{def A(M)}
$(1)$ Let $\ca{M}$ be a Hilbert manifold. 
Suppose that its injectivity radius is greater than $2\varepsilon$ everywhere for some $0<\vep\leq \infty$. 
Let $x_0,x\in \ca{M}$, and suppose that $d(x,x_0)<2\vep$.
Then $x_0$ is contained in the image of $\exp_{x}:B_{2\vep}(T_{x}\ca{M})\to \ca{M}$, and hence $\log_x(x_0)$ is well-defined.
The local Clifford operator at $x_0$ is defined by
$$C_{x_0}(x,t):=
(t,-\log_{x}(x_0))\in T_x\ca{M}\oplus t\bb{R}.$$
Intuitively, it is given by 
$$C_{x_0}(x,t)=(t,\overrightarrow{x_0x}).$$
Note that this $\overrightarrow{x_0x}$ is a tangent vector at $x$.

$(2)$ The local Bott homomorphism $\beta_{x_0}:\ca{S}_\vep\to \Pi_b(\ca{M})$ centered at $x_0\in \ca{M}$ is defined as follows: For $f\in \ca{S}_\vep$,
$$\beta_{x_0}(f)(x,t):=
\begin{cases}
f(C_{x_0}(x,t)) & (d(x,x_0)<\vep) \\
0 & (d(x,x_0)\geq\vep),
\end{cases}$$
where $f(C_{x_0}(x,t))$ is the functional calculus in the $C^*$-algebra $\Cl(t\bb{R}\oplus T_x\ca{M})$.

$(3)$ The $C^*$-algebra $\ca{A(M)}$ is defined to be the $C^*$-subalgebra of $\Pi_b(\ca{M})$ generated by the image of the Bott homomorphisms:
$$\ca{A(M)}:=C^*\bra{\bbra{
\beta_{x_0}(f)\ \middle|\  x_0\in\ca{M}, f\in \ca{S}_\vep}}.$$
\end{dfn}

\begin{rmk}
Note that $\beta_{x_0}(f)$ is continuous.
Since it is supported on the $\vep$-neighborhood of $x_0$, we regard $\beta_{x_0}(f)$ as an infinite-dimensional version of a compactly supported continuous function.
\end{rmk}

It also shares several properties with the $C^*$-algebra of \cite{GWY}.

\begin{pro}[{\cite[Proposition 5.15]{GWY}, \cite[Proposition 4.9]{T3}}]\label{properties of GWY algebra}
$\ca{A(M)}$ is separable whenever $\ca{M}$ is separable.
\end{pro}

In order to simplify several arguments on the Bott maps, we impose the following.

\begin{asm}\label{Katei}
Let $\ca{M}$ be a separable Hilbert manifold all of whose sectional curvatures are bounded above by $\delta$ and the injectivity radius is greater than $2\vep>0$ at each point.
When $\delta>0$, by possibly shrinking $\vep$, we may assume that $\vep<\frac{\pi}{2\sqrt{\delta}}$.
\end{asm}

The following is proved using the Rauch comparison theorem for Hilbert manifolds \cite{Bil}.

\begin{lem}[{\cite[Lemma 5.7]{T2}}]
If $\ca{M}$ satisfies the above assumption, for any $m\in \ca{M}$ and $v_1,v_2\in T_m\ca{M}$ such that $\|v_1\|,\|v_2\|<\vep$, we have
$$d(\exp_m(v_1),\exp_m(v_2))\geq \frac{1}{2}\|v_1-v_2\|.$$
In other words, for $m,z_0,z\in \ca{M}$ such that $d(z_0,m)<\vep$ and $d(z,m)<\vep$, 
$$\|\overrightarrow{z_0m}-\overrightarrow{zm}\|\leq 2d(z_0,z).$$
\end{lem}
\begin{cor}
For every $f\in \ca{S}_\vep$, the map $\ca{M}\ni x_0\mapsto \beta_{x_0}(f)\in \ca{A(M)}$ is norm-continuous.
\end{cor}
\begin{proof}
Let $\Delta(f)=\sum_if_1^i\grotimes f_2^i$.
Approximate $\Delta(f)$ by a finite sum $\sum_{i=1}^Nf_1^i\grotimes f_2^i$.
Suppose that $d(x_0,x_1)<\vep$. For every $m\in \ca{M}$,
$\beta_{x_0}(f)-\beta_{x_1}(f)$
is approximated in norm by the function
$$m\mapsto \sum_{i=1}^Nf_1^i\grotimes \bbra{f_2^i(\vect{x_0m})-f_2^i(\vect{x_1m})}.$$
This is because each $\beta_{x_0}$ is a $*$-homomorphism, and the approximation error is uniformly controlled.
Since each $f_2^i$ is uniformly continuous, and since $\|\vect{x_0m}-\vect{x_1m}\|\leq 2d(x_0,x_1)$,
$\sup_{m\in \ca{M}}\|\beta_{x_0}(f)(m)-\beta_{x_1}(f)(m)\|$ is arbitrarily small if $d(x_0,x_1)$ is sufficiently small.
Since the approximation error can be made arbitrarily small independently of $x_0$, $x_1$ and $m$, we have the result.
\end{proof}
\begin{rmk}
If $\ca{M}$ is non-positively curved, one has $\|\overrightarrow{z_0m}-\overrightarrow{zm}\|\leq d(z_0,z)$ for any $z_0,z\in \ca{M}$.
This fact is used in \cite{GWY}.
\end{rmk}

\medskip

Next, we describe two natural homomorphisms associated to embeddings: a Gysin homomorphism for totally geodesic embeddings and a pullback-type homomorphism for inclusions into a direct product.

In \cite{T3}, we have defined a Gysin map $\tau:\ca{A}(M)\to \ca{A}(LM)$. The construction relies on the fact that the embedding $\iota:M\to LM$ is a totally geodesic isometric embedding.
The construction is as follows.
For $f\grotimes \phi\in \ca{A}(M)$,
\begin{equation}\label{Def of w w m}
\tau(f\grotimes \phi)(t,l)
:=\begin{cases}
f(t\grotimes 1+1\grotimes C^{\perp}_{\iota(x)}(l))
d\iota_x(\phi(x)) 
& (l\in U_\vep(\iota(x))) \\
0 & (l\notin U_\vep(\iota(x))).
\end{cases}
\end{equation}
It gives a $*$-homomorphism $\ca{A}(M)\to \Pi(LM)$, but it is highly non-trivial that the image of $\tau$ actually lies in $\ca{A}(LM)$.
For a totally geodesic submanifold, this is proved as follows.
The subalgebra generated by $\beta_x(f)\in \ca{A}(M)$ for $x\in M$ and $f\in \ca{S}_\vep$ is dense in $\ca{A}(M)$ by the Stone-Weierstrass theorem.
For each generator $\beta_x(f)\in \ca{A}(M)$, we have $\tau(\beta_x(f))=\beta_{\iota(x)}(f)\in \ca{A}(LM)$.
Hence, by continuity, every element $\sum f_i\grotimes \phi_i$ is mapped into $\tau(\sum f_i\grotimes \phi_i) \in \ca{A}(LM)$.
This construction also works for general (possibly infinite-dimensional) totally geodesic submanifolds and we denote it by $\iota_!$ for a totally geodesic embedding $\iota$. 
Although we do not use this construction in this paper, we include the definition for completeness.
One of the main tasks of this paper is to generalize this construction as a $KK$-element.

\begin{pro}
Let $\ca{M}$ be a Hilbert manifold and let $M$ be a totally geodesic (possibly infinite-dimensional) submanifold in $\ca{M}$.
We denote the inclusion $M\to\ca{M}$ by $\iota$.
Then, we can define $\iota_!:\ca{A}(M)\to \ca{A(M)}$ by the formula (\ref{Def of w w m}).
\end{pro}

We also formulate a pullback-type homomorphism for the inclusion into the first factor $\iota:\ca{M}\hookrightarrow \ca{M}\times X$ of the direct product of a Hilbert manifold $\ca{M}$ and a finite-dimensional complete Riemannian manifold $X$.
Then, $\iota$ is clearly a totally geodesic embedding.

$\ca{A}(\ca{M}\times X)$ is generated by $\beta^{\ca{M}\times X}_{(m_0,x_0)}(f)$ for $m_0\in \ca{M}$, $x_0\in X$ and $f\in \ca{S}_\vep$.
The Clifford operator for $\ca{M}\times X$, denoted by $C^{\ca{M}\times X}$, can be written as
$$C^{\ca{M}}\grotimes 1+1\grotimes C^X$$
under the identification $T_{(m,x)}(\ca{M}\times X)\cong T_m\ca{M}\oplus T_xX$.
For $f\in \ca{S}_\vep$, let $\Delta(f)=\sum_if_{1}^i\grotimes f_{2}^i$.
Then, for $t\in (-\vep,\vep)$ and $(m,x)\in \ca{M}\times X$ close enough to $(m_0,x_0)$, we have
\begin{align*}
\beta^{\ca{M}\times X}_{(m_0,x_0)}(f)(t,m,x)
&=f(t\grotimes 1+1\grotimes C^{\ca{M}\times X}_{(m_0,x_0)}(m,x)) \\ &=f(t\grotimes 1\grotimes 1+1\grotimes C^{\ca{M}}_{m_0}(m)\grotimes 1+1\grotimes 1\grotimes C^{X}_{x_0}(x)) \\
&=\Delta(f)(t\grotimes 1+1\grotimes C^{\ca{M}}_{m_0}(m),C^{X}_{x_0}(x)) \\
&=\sum_if_1^i(t\grotimes 1+1\grotimes C^{\ca{M}}_{m_0}(m))
\grotimes f_2^i(C^{X}_{x_0}(x))
\in\ca{A}(\ca{M})\grotimes Cl_\tau(X).
\end{align*}
Thus, every generator of $\ca{A}(\ca{M}\times X)$ belongs to $\ca{A}(\ca{M})\grotimes Cl_\tau(X)$.
Hence, $\ca{A}(\ca{M}\times X)$ can be regarded as a $C^*$-subalgebra of $\ca{A(M)}\grotimes Cl_\tau(X)$.
Therefore, for $x\in X$ and the embedding $\iota_x:\ca{M}\ni m\mapsto (m,x)\in \ca{M}\times X$, we can define $\iota_x^*$ as follows.

\begin{dfn}\label{pullback along dp}
We define $\iota_x^*$ by the composition of the following homomorphisms of $C^*$-algebras
$$\ca{A}(\ca{M}\times X)\to \ca{A}(\ca{M})\grotimes Cl_\tau(X)\xrightarrow{\id\grotimes(\text{evaluation at }x)}
\ca{A(M)}\grotimes \Cl(T_xX).$$
The corresponding $KK$-element is denoted by $[\iota_x^*]\in KK(\ca{A(M}\times X),\ca{A(M)}\grotimes \Cl(T_xX))$.
\end{dfn}

We will use this construction to prove that the homomorphism constructed in this paper is non-trivial.

\section{Geometric $K$-homology and operator $K$-theory}

\subsection{Pushforward map for Hilbert manifolds in $KK$-theory}\label{subsection PF for Hilb}

Let $M$ be a complete Riemannian manifold without boundary, and let $\ca{M}$ be a Hilbert manifold.
Let $f:M\to \ca{M}$ be a continuous map.
We would like to construct a $KK$-element playing the role of a Gysin map $[f_!]$.
Such an element has been constructed and extensively studied in \cite{CS,EM}, but we take an alternative approach in order to deal with infinite-dimensional manifolds.

\begin{dfn}
For a continuous map $f:M\to \ca{M}$, we define
$i^f:M\to \ca{M}\times {M}$ by $i^f(x):=(f(x),x)$.
We define a $*$-homomorphism $(i^f)_!:\ca{A}(M)\to \Pi(\ca{M}\times M)$ by
$$(i^f)_!(g\grotimes \phi)(t,m,x):=
\beta_{f(x)}(g)(t,m)\grotimes \phi(x)=
g(X(t)\grotimes 1+1\grotimes C^{\ca{M}}_{f(x)}(m))
\grotimes \phi(x)
$$
for $t\in (-\vep,\vep)$, $m\in \ca{M}$ and $x\in M$.
\end{dfn}

\begin{pro}
$(i^f)_!$ is a well-defined  $*$-homomorphism from $\ca{A}(M)$ into $ \ca{A(M)}\grotimes Cl_\tau(M)$.
\end{pro}
\begin{proof}
Clearly, $(i^f)_!$ gives a $*$-homomorphism from $\ca{A}(M)$ to the ``outer box'' $\Pi(\ca{M}\times M)$.
It suffices to check that the image is contained in $\ca{A(M)}\grotimes Cl_\tau(M)$.

Let $g\grotimes \phi\in \ca{S}_\vep\grotimes Cl_\tau(M)$.
For arbitrary $\eta>0$, $y_0\in M$ and $m_0\in \ca{M}$, 
we first show that there exists $\delta>0$ such that if $d(y_0,y)<\delta$ and $d(m,m_0)<\vep$, then 
$$\left\|g(t\grotimes 1+1\grotimes \vect{f(y_0)m})-
g(t\grotimes 1+1\grotimes \vect{f(y)m})\right\|<\eta.$$

To prove this,
we first consider an even function $g$.
Let $g(t)=g'(t^2)$. Then, we have
$$g(t\grotimes 1+1\grotimes \vect{f(y_0)m})-
g(t\grotimes 1+1\grotimes \vect{f(y)m})
=g'(t^2+d(f(y_0),m)^2)-g'(t^2+d(f(y),m)^2).$$
Since $g$ and $g'$ are uniformly continuous, it suffices to control $|t^2+d(f(y_0),m)^2-\bbra{t^2+d(f(y),m)^2}|$.
\begin{align*}
&|t^2+d(f(y_0),m)^2-\bbra{t^2+d(f(y),m)^2}| \\
&\ \ \ =|d(f(y_0),m)-d(f(y),m)||d(f(y_0),m)+d(f(y),m)| \\
&\ \ \ \leq d(f(y_0),f(y))|d(f(y_0),m)+d(f(y),m)| \\
&\ \ \ \leq d(f(y_0),f(y))\bbra{d(f(y_0),m_0)+d(f(y),m_0)+2d(m,m_0)} \\
&\ \ \ \leq  d(f(y_0),f(y))\bbra{d(f(y_0),m_0)+d(f(y),m_0)+2\vep} \\
&\ \ \ \leq d(f(y_0),f(y))\bbra{2d(f(y_0),m_0)+d(f(y),f(y_0))+2\vep}.
\end{align*}
Therefore, if $y$ is close enough to $y_0$, 
then $|t^2+d(f(y_0),m)^2-\bbra{t^2+d(f(y),m)^2}|$ becomes arbitrarily small.

For an odd function $g$, let $g(t)=tg'(t^2)$.
\begin{align*}
&\left\|g(t\grotimes 1+1\grotimes \vect{f(y_0)m})-
g(t\grotimes 1+1\grotimes \vect{f(y)m})\right\| \\
&\ \ \ =\left\|
(t\grotimes 1+1\grotimes \vect{f(y_0)m})
g'(t^2+d(f(y_0),m)^2)
-
(t\grotimes 1+1\grotimes \vect{f(y)m})
g'(t^2+d(f(y),m)^2)\right\| \\
&\ \ \ \leq
\left\|
(t\grotimes 1+1\grotimes \vect{f(y_0)m})
\bbra{g'(t^2+d(f(y_0),m)^2)
-g'(t^2+d(f(y),m)^2)}\right\| \\
&\ \ \ \ \ \ +
\left\|(1\grotimes \vect{f(y_0)m}-1\grotimes \vect{f(y)m})
g'(t^2+d(f(y),m)^2)\right\| \\
&\ \ \ \leq 2\vep\left|{g'(t^2+d(f(y_0),m)^2)
-g'(t^2+d(f(y),m)^2)}\right|+
\left\|\vect{f(y_0)m}-\vect{f(y)m}\right\|
|g'(t^2+d(f(y),m)^2)|.
\end{align*}
By Assumption \ref{Katei}, we have $\left\|\vect{f(y_0)m}-\vect{f(y)m}\right\|<2d(f(y_0),f(y))$.
Therefore, if $y$  and $y_0$ are close enough, 
then $\left\|g(t\grotimes 1+1\grotimes \vect{f(y_0)m})-
g(t\grotimes 1+1\grotimes \vect{f(y)m})\right\|$ becomes arbitrarily small.

Thus the map
$$(m,y)\mapsto
g(X\otimes1+1\otimes\vect{f(y)m})$$
is locally uniformly continuous in $y$ and $m$.

We now prove that $(i^f)_!(g\grotimes \phi)\in \ca{A(M)}\grotimes Cl_\tau(M)$ for $g\in \ca{S}_\vep$ and $\phi\in Cl_\tau(M)$.
It suffices to prove the claim for compactly supported $\phi$.
Let $\{U_j\}$ be an open cover of the support of $\phi$ such that the diameter of each $U_i$ is less than $\delta$, and let $\{\rho_j\}$ be a partition of unity associated to this open cover. Let $x_j\in U_j$.
Since the support of $\phi$ is compact, we may assume that the open cover is finite.
We compare $(i^f)_!(g\grotimes \rho_j\phi )$ with
$\beta_{f(x_j)}(g)\grotimes \rho_j\phi \in \ca{A(M)}\grotimes Cl_\tau(M)$.
\begin{align*}
&(i^f)_!(g\grotimes \rho_j\phi )(t,m,x)-
\beta_{f(x_j)}(g)\grotimes \rho_j\phi (t,m,x) \\
&\ \ \ =g\bra{t\grotimes 1+1\grotimes \vect{f(x_j)m}}\grotimes \rho_j(x)\phi (x)
-g\bra{t\grotimes 1+1\grotimes \vect{f(x)m}}\grotimes \rho_j(x)\phi (x) \\
&\ \ \ =\bbra{g\bra{t\grotimes 1+1\grotimes \vect{f(x_j)m}}
-g\bra{t\grotimes 1+1\grotimes \vect{f(x)m}}}\grotimes \rho_j(x)\phi (x).
\end{align*}
Since $\rho_j(x)=0$ unless $d(x,x_j)<\delta$, we may assume $d(x,x_j)<\delta$.
Therefore, 
by the preceding argument, we have
\begin{align*}
&\|(i^f)_!(g\grotimes \rho_j\phi )-
\beta_{f(x_j)}(g)\grotimes \rho_j\phi \| \\
&\ \ \ \leq
\sup_{m\in U_{\vep}(f(x_j)),x\in M}
\left\|g\bra{X\grotimes 1+1\grotimes \vect{{f(x_j)}(m)}}
-g\bra{X\grotimes 1+1\grotimes \vect{f(x)m}}\right\|
\|\rho_j(x)\phi (x)\| \\
&\ \ \ \leq \eta\|\rho_j\phi \|,
\end{align*}
and $(i^f)_!(g\grotimes \rho_j\phi )$ can be approximated by  $\beta_{f(x_j)}(g)\grotimes \rho_j\phi \in \ca{A(M)}\grotimes Cl_\tau(M)$.
Therefore, $(i^f)_!(g\grotimes \phi)=\sum_i(i^f)_!(g\grotimes \rho_j\phi)$ can be approximated by elements of $\ca{A(M)}\grotimes Cl_\tau(M)$.
\end{proof}

Let $[d_M]\in KK(Cl_\tau(M),\bb{C})$ be the Dirac element in \cite{Kas88,Kas15}.
The Kasparov product with $[d_M]$ gives an index map $KK(\bb{C},Cl_\tau(M))\to KK(\bb{C},\bb{C})=\bb{Z}$.

We now define the Gysin map for an infinite-dimensional manifold as follows.

\begin{dfn}
We define $[f_!]\in KK(\ca{A}(M),\ca{A(M)})$ by
$[f_!]:=[(i^f)_!]\grotimes [d_M].$
\end{dfn}

Our Gysin map is homotopy invariant.

\begin{pro}
If continuous maps $f_0,f_1:M\to \ca{M}$ are homotopic, we have
$[(f_0)_!]=[(f_1)_!]$.
\end{pro}
\begin{proof}
A homotopy between continuous maps $f_0$ and $f_1$ gives a homotopy between $*$-homomorphisms $(i^{f_0})_!$ and $(i^{f_1})_!$.
\end{proof}

Our Gysin map satisfies the following functoriality property.

\begin{pro}
Let $M_1$ and $M_2$ be complete Riemannian manifolds without boundary, $\ca{M}$ be a Hilbert manifold satisfying Assumption \ref{Katei}, and $M_1\xrightarrow{f_1} M_2\xrightarrow{f_2} \ca{M}$ be continuous maps.
For simplicity, we assume $f_2$ is uniformly continuous.
Then we have
$${[b_+]}
\grotimes
[(f_1)_!]
\grotimes_{\ca{A}(M_2)}
[(f_2)_!]
={[b_+]}\grotimes
[(f_2\circ f_1)_!]$$
in $KK(Cl_\tau(M_1),\ca{A(M)})$.
\end{pro}

\begin{proof}
Let $j^{f_2}:M_1\times M_2\to \ca{M}\times M_1\times M_2$ be 
$j^{f_2}(m_1,m_2):=(f_2(m_2),m_1,m_2)$.
Note that $j^{f_2}$ can be written as $i^{f_2\circ \rm pr_2}$, where ${\rm pr}_2:M_1\times M_2\to M_2$ is the projection onto the second factor.
Hence $(j^{f_2})_!$ is defined.
Note also that $i^{f_1\times (f_2\circ f_1)}(m_1)=
((f_2\circ f_1)(m_1),f_1(m_1),m_1)$.
Let us consider the following diagram in the Kasparov category:
$$\xymatrix{
&&
\ca{A(M)}\grotimes Cl_\tau(M_1)
\ar@/^50pt/^{-\grotimes[d_{M_1}]}[rrddd] 
\ar@{}[rrddd]^{\big{(1)}}
&& 
\\
&&
\ca{A(M)}\grotimes Cl_\tau(M_2)\grotimes Cl_\tau(M_1) 
\ar_{-\grotimes [d_{M_1}]}[rd]
\ar_{-\grotimes [d_{M_2}]}[u] 
& & 
\\
&
\ca{A}(M_2)\grotimes Cl_\tau(M_1) 
\ar_{-\grotimes [d_{M_1}]}[rd]
\ar_{-\grotimes [(j^{f_2})_!]}[ru] 
\ar@{}[rr]|{\big{(2)}}& &
\ca{A(M)}\grotimes Cl_\tau(M_2) \ar_{-\grotimes [d_{M_2}]}[rd] & 
\\
\ca{A}(M_1)
\ar_{-\grotimes [(i^{f_1})_!]}[ru]
\ar@/^50pt/^{-\grotimes[(i^{f_2\circ f_1})_!]}[rruuu]
\ar@{}[rruuu]^{\big{(3)}}
 & &
\ca{A}(M_2)
\ar_{-\grotimes [(i^{f_2})_!]}[ru] & &
\ca{A(M)}.
}$$

We need to prove that $[b_+]\grotimes [(i^{f_2\circ f_1})_!]\grotimes [d_{M_1}]=
[b_+]
\grotimes [(i^{f^1})_!]\grotimes [d_{M_1}]
\grotimes [(i^{f^2})_!]\grotimes [d_{M_2}]$.
To this end, we first prove that $(1)$ and $(2)$ above commute. Then, we consider $(3)$ after taking the Kasparov product with $[b_+]$, namely, we prove that
$[b_+]\grotimes [(i^{f_2\circ f_1})_!]=
[b_+]\grotimes [(i^{f_1})_!]\grotimes[(j^{f_2})_!]\grotimes [d_{M_2}]$.

The square $(1)$ clearly commutes.

A direct computation shows that $(j^{f_2})_!=(i^{f_2})_!\grotimes \id_{Cl_\tau(M_1)}$, and hence $(2)$ commutes.

For $(3)$, we first show that we may replace $\ca{A}(M_i)$ by  the subalgebras
$\ca{S}_\delta\grotimes Cl_\tau(M_i)$ for $i=1,2$ and $\delta\ll \vep$.
We shall choose $\delta$ later so that the $\delta$-neighborhood of the graph of $f_2$ and that of $f_2\circ f_1$ are sufficiently close to each other.
It is easy to see that the inclusion $k:\ca{S}_\delta\hookrightarrow \ca{S}_\vep$ has a homotopy inverse.
We denote the corresponding inclusions
$$\ca{S}_\delta\grotimes Cl_\tau(M_1)\hookrightarrow \ca{S}_\vep\grotimes Cl_\tau(M_1)=\ca{A}(M_1)\ \text{and}\ 
\ca{S}_\delta\grotimes Cl_\tau(M_2)\hookrightarrow \ca{S}_\vep\grotimes Cl_\tau(M_2)=\ca{A}(M_2)
$$
by the same symbol $k$.
Then the following two diagrams clearly commute:
$$\begin{CD}
\ca{S}_\delta\grotimes Cl_\tau(M_1)
@>{(i^{f_1})_!}>>
\ca{S}_\delta\grotimes Cl_\tau(M_2)\grotimes Cl_\tau(M_1)
@>{(j^{f_2})_!}>>
\ca{A(M)}\grotimes Cl_\tau(M_2)\grotimes Cl_\tau(M_1)
 \\
@VkVV @VVkV @VV=V \\
\ca{A}(M_1) 
@>{(i^{f_1})_!}>> 
\ca{A}(M_2)\grotimes Cl_\tau(M_1)
@>{(j^{f_2})_!}>>
\ca{A(M)}\grotimes Cl_\tau(M_2)\grotimes Cl_\tau(M_1).
\end{CD}$$
$$\begin{CD}
\ca{S}_\delta\grotimes Cl_\tau(M_1)
@>{(i^{f_1\times (f_2\circ f_1)})_!}>>
\ca{A(M)}\grotimes Cl_\tau(M_2)\grotimes Cl_\tau(M_1)
 \\
@VkVV @VV=V \\
\ca{A}(M_1) 
@>{(i^{f_1\times (f_2\circ f_1)})_!}>>
\ca{A(M)}\grotimes Cl_\tau(M_2)\grotimes Cl_\tau(M_1),
\end{CD}$$
and the vertical arrows are invertible in the Kasparov category.

We now verify that $[(i^{f_1})_!]\grotimes[(j^{f_2})_!]
=[(i^{f_1\times (f_2\circ f_1)})_!]$.
Let $g\in \ca{S}_\delta$ and $\phi\in Cl_\tau(M_1)$.
We may assume that $\phi$ is compactly supported.
Let $\Delta(g)=\sum_ig^i_1\grotimes g^i_2$. Then,
$$(i^{f_1})_!(g\grotimes \phi)(t,m_2,m_1)=\sum_ig^i_1(t)g^i_2(\vect{f_1(m_1)m_2})\grotimes \phi(m_1)$$
and hence
$$(j^{f_2})_!(i^{f_1})_!(g\grotimes \phi)(t,x,m_2,m_1)=\sum_ig^i_1(t\grotimes 1+1\grotimes \vect{f_2(m_2)x})\grotimes g^i_2(\vect{f_1(m_1)m_2})\grotimes \phi(m_1).$$
Thanks to the coassociativity of $\Delta$, we have
\begin{align*}
\sum_ig^i_1(t\grotimes 1+1\grotimes \vect{f_2(m_2)x})\grotimes g^i_2(\vect{f_1(m_1)m_2})\grotimes \phi(m_1)
&=
\sum_ig^i_1(t)g^i_2(\vect{f_2(m_2)x}\grotimes  1+1\grotimes \vect{f_1(m_1)m_2})\grotimes \phi(m_1) \\
&=
\sum_ig^i_1(t)\Delta(g^i_2)(\vect{f_2(m_2)x}, \vect{f_1(m_1)m_2})\grotimes \phi(m_1) .
\end{align*}
On the other hand, 
$$(i^{f_1\times (f_2\circ f_1)})_!(g\grotimes \phi)(t,x,m_2,m_1)
=\sum_ig^i_1(t)\Delta(g^i_2)(\vect{f_2\circ f_1(m_1)x}, \vect{f_1(m_1)m_2})\grotimes \phi(m_1).$$
Then, we consider the linear homotopy
$$s\mapsto \sum_ig^i_1(t)\Delta(g^i_2)(s\vect{f_2\circ f_1(m_1)x}+(1-s)\vect{f_2(m_2)x}, \vect{f_1(m_1)m_2})\grotimes \phi(m_1)=:H_s(t,x,m_1,m_2).$$
Strictly speaking, it is defined as follows:
$${\small \begin{cases}
\sum_ig^i_1(t)\Delta(g^i_2)(s\vect{f_2\circ f_1(m_1)x}+(1-s)\vect{f_2(m_2)x}, \vect{f_1(m_1)m_2})\grotimes \phi(m_1)
& (d(f_2\circ f_1(m_1),x)<2\vep \text{ and }
d(f_2(m_2),x)<2\vep) \\
0& (\text{otherwise}).
\end{cases}}$$
We prove that $H_s$ is well-defined and continuous.
It suffices to check that if 
$d(f_2\circ f_1(m_1),x)>3\vep/2$ or 
$d(f_2(m_2),x)>3\vep/2$, we have $H_s=0$.
The crucial observation is that 
$$\Delta(g^i_2)(s\vect{f_2\circ f_1(m_1)x}+(1-s)\vect{f_2(m_2)x}, \vect{f_1(m_1)m_2})=0 \text{ if } d(f_1(m_1),m_2)\geq \delta.
$$
Therefore, by symmetry, it suffices to prove that if $d(f_2\circ f_1(m_1),x)>3\vep/2$  and $d(f_1(m_1),m_2)< \delta$, we have $d(f_2(m_2),x)>\vep$.
Now we choose $\delta>0$ so small that if $d(m_2,m_2')<\delta$, we have 
$d(f_2(m_2),f_2(m_2'))<\vep/2$. 
Thus, if $d(f_1(m_1),m_2)< \delta$, we have $d(f_2(m_2),f_2(f_1(m_1)))<\vep/2$. 
This is possible because $f_2$ is uniformly continuous.
Then, by the triangle inequality,
$$d(f_2(m_2),x)
\geq
d(f_2\circ f_1(m_1),x)-
d(f_2\circ f_1(m_1),f_2(m_2))
>\vep.$$
Therefore $H_s$ is well-defined and continuous.

Finally, we compare $[b_+]\grotimes [(i^{f_2\circ f_1})_!]$ and $[b_+]\grotimes [(i^{f_1\times (f_2\circ f_1)})_!]\grotimes [d_{M_2}]$.
For this, we use the language of $\ca{R}KK$-theory, namely, that of fields of Kasparov modules.
See \cite{Kas88,LeG} for details on this theory.
In general, for a locally compact Hausdorff space $X$, a $C^*$-algebra $A$ is called a $C_0(X)$-algebra if it is equipped with a nondegenerate structure homomorphism from $C_0(X)$ to the center of the multiplier algebra of $A$.
If $A$ is a $C_0(X)$-algebra, it is known that there exists a field of $C^*$-algebras $\{A_x\}_{x\in X}$ such that $A$ is the ``section algebra'' of this field \cite{Dix,Blan}.
In the same way, a Kasparov module over $X$ can be described as a field of Kasparov modules parametrized by $X$.

Then, $\ca{A}(M_1)$, $\ca{A}(M_2)\grotimes Cl_\tau(M_1)$ and 
$\ca{A}(\ca{M})\grotimes Cl_\tau(M_2\times M_1)$ correspond, respectively, to the field 
$\{\ca{S}_\vep\grotimes \Cl(T_{m_1}M_1)\}_{m_1\in M_1}$, 
$\{\ca{A}(M_2)\grotimes \Cl(T_{m_1}M_1)\}
_{m_1\in M_1}$ and $\{\ca{A(M)}\grotimes Cl_\tau(M_2)\grotimes \Cl(T_{m_1}M_1)\}_{m_1\in M_1}$.

We describe $(i^{f_2\circ f_1})_!$ and $(i^{f_1\times (f_2\circ f_1)})_!$ in terms of fields of Kasparov modules.
First, $(i^{f_2\circ f_1})_!$ corresponds to 
$$\bbra{\ca{A(M)}\grotimes \Cl(T_{m_1}M_1),\beta_{f_2\circ f_1(m_1)}\grotimes \id,0}_{m_1\in M_1}.$$
In order to describe $(i^{f_1\times (f_2\circ f_1)})_!$,
we define $\beta^0_{m}:\ca{S}_\vep\to Cl_\tau(M_2)$ by $ \beta^0_{m}(g)(m'):=g(\vect{mm'})$ for $m,m'\in M_2$.
Then, $(i^{f_1\times (f_2\circ f_1)})_!$ corresponds to 
\begin{align*}
&\bbra{\ca{A(M)}\grotimes Cl_\tau(M_2)\grotimes \Cl(T_{m_1}M_1),
(\beta_{f_2\circ f_1(m_1)}\grotimes \beta^0_{f_1(m_1)}\grotimes \id)
\circ 
(\Delta\grotimes \id),0}_{m_1\in M_1}
\end{align*}
This field is clearly the Kasparov product of 
$[\Delta]\grotimes_{\bb{C}}[{\bf 1}_{Cl_\tau(M_1)}]\in \ca{R}KK(M_1;\ca{A}(M_1),\ca{S}_\vep\grotimes \ca{A}(M_1))$ and 
\begin{align*}
&\bbra{\ca{A(M)}\grotimes Cl_\tau(M_2)\grotimes \Cl(T_{m_1}M_1),
\beta_{f_2\circ f_1(m_1)}\grotimes \beta^0_{f_1(m_1)}\grotimes \id,0}_{m_1\in M_1}.
\end{align*}
The latter represents an element of $\ca{R}KK(M_1;\ca{S}_\vep\grotimes \ca{A}(M_1),\ca{A(M)}\grotimes Cl_\tau(M_2\times M_1))
$.
Then, under the isomorphism 
$$\ca{A(M)}\grotimes Cl_\tau(M_2\times M_1)
\cong 
\bbra{\ca{A(M)}\grotimes Cl_\tau(M_1)}\grotimes_{C(M_1)}\bbra{Cl_\tau(M_2)\grotimes C(M_1)},$$ 
the above field is the external Kasparov product over $M_1$ of
\begin{align*}
&\bbra{\ca{A(M)}\grotimes \Cl(T_{m_1}M_1),
\beta_{f_2\circ f_1(m_1)}\grotimes \id,0}_{m_1\in M_1}
\end{align*}
which is precisely the field corresponding to $(i^{f_2\circ f_1})_!$, and
\begin{align*}
&\bbra{Cl_\tau(M_2),
\beta^0_{f_1(m_1)},0}_{m_1\in M_1}.
\end{align*}
We denote the corresponding $\ca{R}KK$-element to this Kasparov module by 
$$[\{\beta^0_{f_1}\}]\in \ca{R}KK(M_1;\ca{S}_\vep \grotimes C(M_1),Cl_\tau(M_2)\grotimes C(M_1)).$$

Thus, $(i^{f_1\times (f_2\circ f_1)})_!$ represents the Kasparov product
$$\bbra{[\Delta]\grotimes[\bm{1}_{Cl_\tau(M_1)}]}
\grotimes_{\ca{S}_\vep\grotimes\ca{A}(M_1)}
\bbra{[(i^{f_2\circ f_1})_!]\grotimes_{C(M_1)}[\{\beta^0_{f_1}\}]}.$$

Taking the Kasparov product with $[b_+]$, we have
\begin{align*}
&[b_+]\grotimes \bbra{[\Delta]\grotimes[\bm{1}_{Cl_\tau(M_1)}]}
\grotimes_{\ca{S}_\vep\grotimes\ca{A}(M_1)}
\bbra{[(i^{f_2\circ f_1})_!]\grotimes_{C(M_1)}[\{\beta^0_{f_1}\}]} \\
&=\bbra{[b_+]\grotimes [b_+]\grotimes [\bm{1}_{Cl_\tau(M_1)}]}
\grotimes _{\ca{S}_\vep\grotimes\ca{A}(M_1)}
\bbra{[(i^{f_2\circ f_1})_!]\grotimes_{C(M_1)}[\{\beta^0_{f_1}\}]} \\
&=\bbra{[b_+]\grotimes_{\ca{S}_\vep} [(i^{f_2\circ f_1})_!]}
\grotimes_{C(M_1)}
\bbra{[b_+]\grotimes_{\ca{S}_\vep}[\{\beta^0_{f_1}\}]}
\end{align*}

Then, a direct computation shows that $[b_+]\grotimes_{\ca{S}_\vep}[\{\beta^0_{f_1}\}]$ is given by 
$\bbra{
Cl_\tau(M_2),1,\frac{C_{f_1(m_1)}^{M_2}}{\vep}
}_{m_1\in M_1}$ and it is precisely the pullback of the local Bott element on $M_2$ by $f_1$.
For the definition of local Bott element, see \cite{Kas88,Kas15}.
Since the Kasparov product of the local Bott element and the Dirac element is $[\bm{1}_{M_2}]$ as shown in \cite{Kas88}, we have
$$[b_+]\grotimes [(i^{f_1\times (f_2\circ f_1)})_!]\grotimes [d_{M_2}]
=[b_+]\grotimes_{\ca{S}_\vep} [(i^{f_2\circ f_1})_!].$$
Combining with the preceding argument, we obtain the result.
\end{proof}

We have defined the Gysin map in the language of $KK$-theory.
We now compare it with the classical $K$-theory Gysin map
$$f_!:K_*(Cl_\tau(X))\to K_*(Cl_\tau(Y))$$
for a smooth map $f:X\to Y$ between finite-dimensional manifolds $X$ and $Y$
in \cite{ASi1,Fur}.
The Gysin map for topological $K$-theory is defined as follows:
\begin{itemize}
\item Factor $f$ as the composition of a closed embedding $i$ and a fiber bundle projection $p$.
\item Define $(i^f)_!$ as the composition of the Thom isomorphism and the extension by zero.
\item Define $p_!$ as the fiber index homomorphism.
\item $f_!:=p_!\circ (i^f)_!$.
\end{itemize}

We follow essentially the same recipe in this paper, except that we specify the embedding and the projection; 
tensor with ${\ca{S}_\vep}$; 
and construct the Thom isomorphism (bundle version of the Bott periodicity) in a way inspired by \cite{HKT}, 
where the Bott periodicity is formulated as the $*$-homomorphism
$$\beta:\ca{S}\ni g\mapsto g(X\grotimes 1+1\grotimes C)\in \ca{S}\grotimes Cl_\tau(\bb{R}^n).$$

If $\ca{M}$ is finite-dimensional, our $[f_!]$ is consistent with the classical one in the following sense.
The classical Gysin map is denoted by $f_!^{cl}$.

\begin{pro}
For $f:M\to {N}$ a smooth map between closed finite-dimensional Riemannian manifolds, and a $p$-graded Clifford bundle $E$ over $M$, we have
$$\bra{[b_+]\grotimes_{\bb{C}} 
[E]}
\grotimes_{\ca{A}(M)} 
[f_!]\grotimes_{\ca{S}_\vep} 
[\ev_0]=f_!^{cl}[E].$$
\end{pro}
\begin{proof}
Since $[f_!]=[(i^f)_!]\grotimes [d_M]$
and the map $u\mapsto u\grotimes [d_M]$ corresponds to the fiber index map, it suffices to prove that 
$\bra{[b_+]\grotimes_{\bb{C}} 
[E]}\grotimes [(i^f)_!]\grotimes_{\ca{S}_\vep} 
[\ev_0]
=(i^f)_!^{cl}([E])$.

Let us compute the left-hand side.
$[(i^f)_!]\grotimes_{\ca{S}_\vep} 
[\ev_0]$ is represented by
$\bra{Cl_\tau({N})\grotimes Cl_\tau(M),\ev_0\circ (i^f)_!,0}$,
and ${[b_+]\grotimes_{\bb{C}} 
[E]}$ is represented by
$\bra{\ca{S}_\vep\grotimes C(M,E),1,\frac{X}{\vep}\grotimes 1}$.
The Kasparov product of them is represented by, in the language of fields of Kasparov modules,
$$\bbra{Cl_\tau(U_\vep(f(m)))\grotimes E_m,1,
\frac{C_{f(m)}^{N}}{\vep}\grotimes 1}_{m\in M},$$
where $U_\vep(f(m))$ is the $\vep$-neighborhood of 
$f(m)$ in ${N}$.

The right-hand side is the tensor product of $[p_2^*(E)]\in K_p(C({N})\grotimes Cl_\tau( M))$ and the extension by zero of the Thom class of the normal bundle of the embedding $i^f:M\to {N}\times M$. 
Since the Thom class is represented by the field of Kasparov modules
$\bbra{\bra{Cl_\tau(U_\vep(f(m))),1,\frac{C^{{N}}_{f(m)}}{\vep}\grotimes 1}_{m\in M}}$, the above Kasparov product equals the right-hand side.
\end{proof}

\subsection{A Poincar\'e duality homomorphism for Hilbert manifolds}

In this subsection, we define a correspondence assigning an element of the operator $K$-theory group to each geometric $K$-homology cycle, and we prove that this correspondence induces a homomorphism from the geometric $K$-homology group of a Hilbert manifold $\ca{M}$.

We briefly review geometric $K$-homology. For simplicity, we do not deal with relative $K$-homology.
See \cite{BD,Jak,BHS} for details.

\begin{dfn}\label{Def of K-homology}
For a topological space $\ca{X}$, a
$K$-homology cycle for $\ca{X}$ is a triple $(M,E,f)$ consisting of:
\begin{itemize}
\item A closed manifold without boundary $M$ equipped with a $Spin^c$-structure.
\item A smooth Hermitian vector bundle $E$ over $M$.
\item A continuous map $f:M\to \ca{X}$.
\end{itemize}
\end{dfn}

The sum of two $K$-homology cycles is defined by the disjoint union: $(M_1,E_1,f_1)+(M_2,E_2,f_2):=(M_1\cup M_2,E_1\cup E_2,f_1\cup f_2)$.
Geometric $K$-homology group $K_*^{geo}(\ca{X})$ is the set of equivalence classes of $K$-homology cycles for $\ca{X}$.
The equivalence relation is generated as follows.
For details, consult \cite{BHS}.

\begin{dfn}[\cite{Jak,BHS}]\label{Equiv rel for K-hom}
$(1)$ For a triple $(M,E_1,f)$ and $(M,E_2,f)$, we have $(M,E_1,f)+(M,E_2,f)\sim(M,E_1\oplus E_2,f)$.

$(2)$ If $M$ is the boundary of a $Spin^c$-manifold $L$, and the $Spin^c$-structure on $M$ is induced from $L$, and $E$ and $f$ extend to $L$, we have $(M,E,f)\sim 0$. This relation is called the {\bf bordism} and such a $K$-homology cycle is said to be null-bordant.

$(3)$ Let $(M,E,f)$ be a $K$-homology cycle.
Let $W$ be an even rank $Spin^c$-vector bundle over $M$.
Consider the total space $Z$ of the unit sphere bundle of $W\oplus \bm{1}$, and the dual $F$ of the even-graded part of the reduced spinor bundle for the vertical tangent bundle over $Z$.
Let $\varpi:Z\to M$ be the natural projection.
Then, we have $(M,E,f)\sim (Z,\varpi^*E\grotimes F,f\circ \varpi)$.
This relation is called the {\bf bundle modification}.
\end{dfn}

The additive inverse of $(M,E,f)$ is given by the same triple equipped with the opposite $Spin^c$-structure, by the bordism relation.

The geometric $K$-homology is $\bb{Z}_2$-graded by the parity of the dimension of $M$.

\begin{dfn}
The group of equivalence classes of $K$-homology cycles for $\ca{X}$ is denoted by $K^{geo}(\ca{X})$.
$K^{geo}_{\rm ev}(\ca{X})$ and $K^{geo}_{\rm odd}(\ca{X})$ are the subgroups of $K^{geo}(\ca{X})$ consisting of equivalence classes of $K$-homology cycles $(M,E,f)$ for which every connected component of $M$ is even-dimensional and odd-dimensional, respectively.
\end{dfn}

We define an assignment that associates to each geometric $K$-homology cycle for $\ca{M}$ an element of $KK^*(\bb{C},\ca{A(M)})$. We will show that, after passing to the $\bb{Z}_2$-grading via formal Bott periodicity, this induces a homomorphism
$$K^{geo}_{\ev}(\ca{M})\to K^{\ev}(\bb{C},\ca{A(M)})\text{ and }
K^{geo}_{\odd}(\ca{M})\to K^{\odd}(\bb{C},\ca{A(M)}).$$
This assignment is analogous to the Poincar\'e duality homomorphism $H_*(X)\to H^{\dim(X)-*}_c(X;o_X)$ for a non-compact manifold with the orientation sheaf $o_X$.

We now make the following definition.

\begin{dfn}
$(1)$ Let $(M,E,f)$ be a $K$-homology cycle for $\ca{M}$.
Assume that $M$ is connected, and let $m:=\dim (M)$.
Let $S_M$ be the $m$-graded spinor bundle over $M$.
We define a corresponding $K$-theory element $\Phi(M,E,f)\in KK(\Cl(\bb{R}^m),\ca{A(M)})
=KK^m(\bb{C},\ca{A(M)})$ by
$$\Phi(M,E,f):=
\bra{[b_+]\grotimes_{\bb{C}} [E\grotimes S_M]}\grotimes[f_!]
=
\bra{[b_+]\grotimes_{\bb{C}} [E\grotimes S_M]}\grotimes
[(i^f)_!]\grotimes
[d_M]\in KK^m(\bb{C},\ca{A(M)}).$$

$(2)$ Let $(M_i,E_i,f_i)$ be $K$-homology cycles for $\ca{M}$ with connected $M_i$'s.
For $(M,E,f)=\sum_i(M_i,E_i,f_i)$, we define
$$\Phi(M,E,f)=\sum_i\Phi(M_i,E_i,f_i).$$
\end{dfn}

In order to prove that the above assignment descends to a homomorphism after passing to the $\bb{Z}_2$-grading, it suffices to deal with three relations: addition, bordism and bundle modification.

\begin{pro}
$\Phi(M,E_1\oplus E_2,f)=\Phi(M,E_1,f)+\Phi(M,E_2,f)$.
\end{pro}
\begin{proof}
This is clear from the definition of addition in $KK(\bb{C},\ca{A}(M))$.
\end{proof}

\begin{pro}
If $(M,E,f)$ is a null-bordant $K$-homology cycle for $\ca{M}$, we have $\Phi(M,E,f)=0$.
\end{pro}
\begin{proof}
Let $m:=\dim(M)$.
Let $L$ be a $Spin^c$-manifold such that $\partial L=M$, and both  $E$ and $f$ extend to $L$. Let $\wt{f}:L\to \ca{M}$ and $\wt{E}$ be the corresponding extensions, and let $\mathring{L}$ be the interior of $L$.
Let $I=[-1,1]$, and let $\mathring{I}=(-1,1)$ denote its interior.
By the collar neighborhood theorem, we identify $\mathring{L}$ with $L\sqcup \bra{M\times [-1,1)}/(L\sim M\times\{-1\})$.
Let $\iota:M\times \mathring{I}\hookrightarrow \mathring{L}$ be the natural inclusion.
We define a homomorphism $\delta^*:KK(\bb{C},C(M))\to KK(\Cl(\bb{R}),C_0(\mathring{L}))$ as follows.
Let $[\tau_{\mathring{I}}]$ be the Bott element associated to $\mathring{I}$:
$$\bra{C_0(\mathring{I},\Cl(\bb{R})),
\gamma,
\sqrt{-1}
xe_1^*}\in KK^1(\bb{C},C_0(\mathring{I}))
=KK(\Cl(\bb{R}),C_0(\mathring{I})),$$
where $\gamma$ is the left multiplication of $\Cl(\bb{R})$, $x$ is the coordinate of $\mathring{I}$, and $e_1^*$ is the graded right multiplication given by $e_1^*w=(-1)^{|w|}we_1$.
Let $\iota_*:C_0(M\times \mathring{I})\to C_0(\mathring{L})$ be the extension by zero.
Then, $\delta^*$ is defined by the composition
$$KK(\bb{C},C(M))\xrightarrow{-\grotimes_{\bb{C}}[\tau_{\mathring{I}}]}
KK(\Cl(\bb{R}),C_0(M\times \mathring{I}))\xrightarrow{\iota_*}
KK(\Cl(\bb{R}),C_0(\mathring{L})).$$
It plays the role of the connecting homomorphism. 
In fact, we have
$$\delta^*(i^*(\wt{E}))
=\iota_*(\wt{E}|_{M\times \mathring{I}}\grotimes [{\rm pr}_2^*\tau_{\mathring{I}}])
=\iota_*( [{\rm pr}_2^*\tau_{\mathring{I}}])\grotimes \wt{E},
$$
where ${\rm pr}_2$ is the projection onto the second factor.
We now prove that $\iota_*([{\rm pr}_2^*\tau_{\mathring{I}}])=0$ as follows.
Let 
$$\widetilde{x}(l):=
\begin{cases}
x({\rm pr}_2(l)) & (l\in M\times \mathring{I}) \\
-1 & (\text{otherwise}).
\end{cases}$$
This is clearly continuous.
Then, $\iota_*([{\rm pr}_2^*\tau_{\mathring{I}}])$ can be represented by
$$(C_0(\mathring{L},\Cl(\bb{R})),\gamma,\sqrt{-1}\widetilde{x}e_1^*)$$
Consider the operator homotopy
$$s\mapsto F_s:=(1-s)\sqrt{-1}\widetilde{x}e_1^*+s\sqrt{-1}e_1^*.$$
Since each $F_s$ is invertible near the boundary, the family $\{F_s\}_{0\leq s\leq 1}$ define a homotopy of Kasparov modules.
Since $F_1=\sqrt{-1}e_1^*$ is invertible on $\mathring{L}$, the corresponding $KK$-element is trivial.
Thus, we have $\delta^*i^*[\widetilde{E}]=0$.

Let $S_{\mathring{I}}=\mathring{I}\times \Cl(\bb{R})$ be the $1$-graded trivial spinor bundle on $\mathring{I}$.
It defines an element
$[S_{\mathring{I}}]\in KK(\Cl(\bb{R})\grotimes C_0(\mathring{I}),Cl_\tau(\mathring{I})).$

Let us consider the following diagram.

$$
\xymatrix{
KK^0(\bb{C},C(M))
 \ar^{-\grotimes [\tau_{\mathring{I}}]}[r] 
 \ar_{-\grotimes [b_+]\grotimes [S_M]}[d] 
 \ar@{}[rd]|{(1)}&
KK^1(\bb{C},C_0(M\times \mathring{I}))
 \ar^{-\grotimes[\iota_*]}[r] 
 \ar|{-\grotimes [b_+]\grotimes[S_L|_{M\times\mathring{I}}]}[d]
 \ar@{}[rd]|{(2)} &
KK^1(\bb{C},C_0(\mathring{L}))
 \ar^{-\grotimes [b_+]\grotimes[S_L]}[d] \\
KK^{m}(\bb{C},\ca{A}(M))
 \ar^{-\grotimes [\tau_{\mathring{I}}]\grotimes[S_{\mathring{I}}]}[r] 
 \ar_{-\grotimes[(i^f)_!]}[d]
 \ar@{}[rd]|{(3)} &
KK^{m+2}(\bb{C},\ca{A}(M\times \mathring{I}))
 \ar^{-\grotimes [\iota_*]}[r] 
 \ar|{-\grotimes[(i^{\wt{f}})_!]}[d]
 \ar@{}[rd]|{(4)} &
KK^{m+2}(\bb{C},\ca{A}(\mathring{L}))
 \ar^{-\grotimes[(i^{\wt{f}})_!]}[d] \\
KK^{m}(\bb{C},\ca{A(M)}\grotimes Cl_\tau(M))
 \ar^{-\grotimes [\tau_{\mathring{I}}]\grotimes[S_{\mathring{I}}]}[r] 
 \ar_{-\grotimes[d_M]}[d] 
 \ar@{}[rrd]|{(5)}&
KK^{m+2}(\bb{C},\ca{A(M)}\grotimes Cl_\tau(M\times \mathring{I})) 
 \ar^{-\grotimes[\iota_*]}[r] &
KK^{m+2}(\bb{C},\ca{A(M)}\grotimes Cl_\tau(\mathring{L}))  
 \ar^{-\grotimes[d_{\mathring{L}}]}[d] \\
KK^{m}(\bb{C},\ca{A(M)})
 \ar_{-\grotimes[\Delta^2]}[rr] & &
KK^{m+2}(\bb{C},\ca{A(M)}) 
}
$$

We prove that all squares commute.

$(1)$ It commutes because $S_L|_{M\times \mathring{I}}=\varpi^*S_M\grotimes S_{\mathring{I}}$, where $\varpi:M\times \mathring{I}\to M$ is the natural projection.

$(2)$ It clearly commutes.

$(3)$ We define $(i^{\widetilde{f}})_!$ as follows. Let $L'=L\cup_ML$, and $F:=\widetilde{f}\cup_M \widetilde{f}$.
Then, $(i^{\widetilde{f}})_!:=(i^{F})_!|_{\ca{A}(\mathring{L})}$.
Define $\widetilde{f}':M\times\mathring{I}\to\ca{M}$ by $\wt{f}'(m,t)=f(m)$.
Note that $\widetilde{f}'$ extends to $L'$, $\wt{f}|_{M\times\mathring{I}}$ is homotopic to $\wt{f}'$, and this homotopy extends to $L'$. 
Since the correspondence $f\mapsto [(i^f)_!]$ is homotopy invariant, we have
	$$[(i^{\wt{f}})_!]=[(i^{\wt{f}'})_!].$$
Then, under the identifications
$$\ca{A}(M\times \mathring{I})\cong \ca{A}(M)\grotimes Cl_\tau(\mathring{I})$$
and
$$\ca{A(M)}\grotimes Cl_\tau(M\times \mathring{I})
\cong
\ca{A(M)}\grotimes Cl_\tau(M)\grotimes Cl_\tau(\mathring{I}),$$
$(i^{\wt{f}'})_!$ corresponds to $(i^{f})_!\grotimes \id$ 
Therefore, the square commutes.

$(4)$ All arrows are induced by $*$-homomorphisms. This square commutes at the level of homomorphisms.

$(5)$ We first comment on the definition of $[d_{\mathring{L}}]$ since $\mathring{L}$ is not complete.
In order to define it, we consider $\widehat{L}=L\cup(M\times [1,\infty))$ the manifold with a cylindrical end.
It is complete and we can define the Dirac element $[d_{\widehat{L}}]$.
We define $[d_{\mathring{L}}]$ by the pullback of it by the extension by zero
$k_*:Cl_\tau(\mathring{L})\hookrightarrow Cl_\tau(\widehat{L})$. 
$$[\iota_*]\grotimes_{Cl_\tau(\mathring{L})}[d_{\mathring{L}}]
=[d_{M\times \bb{R}}|_{M\times \mathring{I}}]
=[d_M]\grotimes_{\bb{C}}[d_\bb{R}|_{ \mathring{I}}]
.$$
Therefore, 
$$[\tau_{\mathring{I}}]\grotimes[S_{\mathring{I}}]
\grotimes[\iota_*]
\grotimes 
[d_{\mathring{L}}]
=[\tau_{\mathring{I}}]\grotimes[S_{\mathring{I}}]
\grotimes([d_M]\grotimes_{\bb{C}}[d_{\bb{R}}|_{ \mathring{I}}])
=[d_M]\grotimes
\bbra{[\tau_{\bb{R}}]\grotimes[S_{\bb{R}}]
\grotimes[d_{\bb{R}}]}.
$$
Then, $[\tau_{\bb{R}}]\grotimes[S_{\bb{R}}]
\grotimes[d_{\bb{R}}]$ is represented by
$$
\bra{\Cl(\bb{R})\grotimes \Cl(\bb{R})\grotimes L^2(\bb{R}),
\gamma_1\grotimes \gamma_2\grotimes \pi,
\sqrt{-1}e_1^*\grotimes 1\grotimes x
+1\grotimes e_2^*\grotimes \frac{d}{dx}
},$$
where $e_2$ is the basis of the second $\bb{R}$.
Then, by a direct computations in $\Cl(\bb{R})\grotimes \Cl(\bb{R})$, the kernel of this operator is spanned by
$$(1+\sqrt{-1}e_1e_2)\grotimes e^{-\frac{x^2}{2}}, \text{ and }
(e_1+\sqrt{-1}e_2)\grotimes e^{-\frac{x^2}{2}}.$$
Since $\Cl(\bb{R})\grotimes \Cl(\bb{R})\cong \CL{2}$ acts on this kernel and the kernel is $2=2^{\frac{2}{2}}$-dimensional, it is a spinor module.
One can directly check that the grading of this representation is given by $\sqrt{-1}e_1e_2$, and hence 
$[\tau_{\bb{R}}]\grotimes[S_{\bb{R}}]
\grotimes[d_{\bb{R}}]$
is equivalent to the canonical generator $[\Delta^2]$.

Therefore, for a null-bordant $K$-homology cycle $(M,E,f)$,
we have 
$$[\Delta^2]\grotimes_{\bb{C}}\Phi(M,E,f)
=\delta^*([E])\grotimes [b_+]\grotimes [S_L]\grotimes [(i^{\widetilde{f}})_!]\grotimes [d_{\mathring{L}}].$$
Since $[E]=i^*[\widetilde{E}]$ and $\delta^*\circ i^*=0$, we obtain the result.
\end{proof}

\begin{pro}
Let $(M,E,f)$ be a $K$-homology cycle for $\ca{M}$, and let $W$ be a $Spin^c$-vector bundle of rank $2r$ over $M$.
Let $(Z,\varpi^*E\otimes F_W,f\circ \varpi)$ be the bundle modification described in Definition \ref{Equiv rel for K-hom} $(3)$. 
Then, we have $\Phi(Z,\varpi^*E\otimes F_W,f\circ \varpi)=[\Delta^2]^r\grotimes_{\bb{C}}\Phi(M,E,f)$.
\end{pro}
\begin{proof}
We take the same strategy as the previous proposition.
Let $m:=\dim(M)$.
We introduce two $KK$-elements.

First, let $S^{\fib}$ be the $2r$-multigraded spinor bundle for the vertical tangent bundle of $Z$ with the multigading $\gamma:\CL{2r}\to \End(S^{\fib})$.
We define $\wt{\varpi}^*:Cl_\tau(M)\to C(Z,Cl(\varpi^*TM))\hookrightarrow Cl_\tau(Z)$ as follows: for $\psi\in Cl_\tau(M)$, 
$$\wt{\varpi}^*(\psi)(z):=\psi(\varpi(z))\in \Cl(T_{\varpi(z)}M)\subseteq \Cl(T_zZ),$$
where $T_{\varpi(z)}M$ is contained in $T_zZ$ as the orthogonal complement of the vertical tangent space.
Then, by an abuse of notation, we define
$[\wt{\varpi}^*]\in KK^{2r}(Cl_\tau(M),Cl_\tau(Z))$ by 
$$\bra{
C(Z,Cl(\varpi^*TM)\grotimes S^{\fib}),\gamma\grotimes \varpi^*,0}\in 
KK( \CL{2r}\grotimes Cl_\tau(M),Cl_\tau(Z))
=KK^{2r}(Cl_\tau(M),Cl_\tau(Z))
.$$
As explained in the paragraph following Definition 4.11 of \cite{BHS}, $S^{\fib}$ is of the form $S^{\fib}_{red}\grotimes \Delta^{2r}$, where $S^{\fib}_{red}$ is the reduced spinor bundle of the vertical tangent bundle of $Z$.
Thus, $[\wt{\varpi}^*]$ is represented by the external tensor product of $[\Delta^2]^r$ and
$$\bra{
C(Z,Cl(\varpi^*TM)\grotimes S^{\fib}_{\red}),\varpi^*,0}\in 
KK^0(Cl_\tau(M),Cl_\tau(Z)).$$

Second, for the bundle $F_W$, we define $[\wt{F_W}]\in KK(Cl_\tau(Z),Cl_\tau(Z))$ by 
$(C(Z,F_W\grotimes \Cl(TZ)),\id\grotimes \pi,0)$, where $\pi$ is the representation of $Cl_\tau(Z)$ by left multiplication.

We now prove that the following diagram commutes.
Note that the composition of the leftmost vertical arrows is $[E]\mapsto \Phi(M,E,f)$, and the composition of the rightmost vertical arrows is $[E']\mapsto \Phi(Z,E',f\circ \varpi)$.

$$
\xymatrix{
KK^{0}(\bb{C},C(M))
 \ar^{-\grotimes[\varpi^*]}[r] 
 \ar_{-\grotimes[b_+]\grotimes [S_M]}[d] 
 \ar@{}[rd]|{(1)}
 &
KK^{0}(\bb{C},C(Z))
 \ar^{-\grotimes[F_W]}[r] 
 \ar|{-\grotimes [b_+]\grotimes[S_Z]}[d] 
 \ar@{}[rd]|{(2)}
 &
KK^{0}(\bb{C},C(Z))
 \ar^{-\grotimes[b_+]\grotimes [S_Z]}[d]
  \\
KK^{m}(\bb{C},\ca{A}(M))
 \ar^{-\grotimes[\wt{\varpi}^*]}[r] 
 \ar_{-\grotimes [(i^f)_!]}[d] 
 \ar@{}[rd]|{(3)}
 &
KK^{m+2r}(\bb{C},\ca{A}(Z))
 \ar^{-\grotimes[\wt{F_W}]}[r] 
 \ar|{-\grotimes [(i^{f\circ \varpi})_!]}[d] 
 \ar@{}[rd]|{(4)}
 &
KK^{m+2r}(\bb{C},\ca{A}(Z))
 \ar^{-\grotimes [(i^{f\circ \varpi})_!]}[d]
  \\
KK^{m}(\bb{C},\ca{A(M)}\grotimes Cl_\tau(M))
 \ar_{-\grotimes([\bm{1}_{\ca{A(M)}}]\grotimes_{\bb{C}} [\wt{\varpi}^*])}[r] 
 \ar_{-\grotimes [d_M]}[d] 
 \ar@{}[rrd]|{(5)}
 &
KK^{m+2r}(\bb{C},\ca{A(M)}\grotimes Cl_\tau(Z))
 \ar_{-\grotimes[\wt{F_W}]}[r] 
 &
KK^{m+2r}(\bb{C},\ca{A(M)}\grotimes Cl_\tau(Z))
 \ar^{-\grotimes [d_Z]}[d]
  \\
KK^{m}(\bb{C},\ca{A(M)})
 \ar_{[\Delta^2]^r}[rr]
 &&
KK^{m+2r}(\bb{C},\ca{A(M)})
}$$

$(1)$ By the definition on the $Spin^c$-structure on $Z$, we have $S_Z\cong \varpi^*S_M\grotimes S^{\fib}$, and the square commutes.

$(2)$ Immediate.

$(3)$ We compute the two Kasparov products separately.
The Hilbert module representing $[(i^f)_!]\grotimes ([\bm{1}_{\ca{A(M)}}]\grotimes_{\bb{C}} [\wt{\varpi}^*])$ is
$$(\ca{A(M)}\grotimes Cl_\tau(M))\grotimes_{Cl_\tau(M)}C(Z,Cl(\varpi^*TM)\grotimes S^{\fib})
\cong \ca{A(M)}\grotimes C(Z,Cl(\varpi^*TM)\grotimes S^{\fib}).$$
The isomorphism is given by 
$(a\grotimes b)\grotimes c\mapsto a\grotimes \wt{\varpi}^*(b)c$.
Thus, the action of $g\grotimes \phi\in \ca{A}(M)$ is given by  multiplication by the function 
\begin{align*}
\id\grotimes \wt{\varpi}^*((i^f)_!(g\grotimes \phi))(t,x,z)
&=g(t\grotimes \id+\id\grotimes \vect{f(\varpi(z))x})
\grotimes \phi(\varpi(z))\\
&=(i^{f\circ \varpi})_!\circ \wt{\varpi}^*(g\grotimes \phi)(t,x,z)
\end{align*}
for $t\in (-\vep,\vep)$, $x\in \ca{M}$, and $z\in Z$.
Hence the two Kasparov modules are isomorphic.
Therefore, we have $[(i^f)_!]\grotimes ([\bm{1}_{\ca{A(M)}}]\grotimes_{\bb{C}} [\wt{\varpi}^*])=
[\wt{\varpi}^*]\grotimes [(i^{f\circ \varpi})_!]$.

$(4)$ Using the language of fields of Kasparov modules, the $KK$-element
$[(i^{f\circ \varpi})_!]$ is represented by the field
$$\{\ca{A(M)}\grotimes \Cl(T_zZ),\beta_{f\circ \varpi(z)}\grotimes \gamma,0\}_{z\in Z}\in 
\ca{R}KK(Z;\ca{S}_\vep\grotimes Cl_\tau(Z),\ca{A(M)}\grotimes Cl_\tau(Z)),$$
where $\gamma$ denotes Clifford multiplication. 
On the other hand, $[\wt{F_W}]$ is represented by the field
$$\{F_W|_z\grotimes \Cl(T_zZ),\id\grotimes \gamma,0\}_{z\in Z}
\in \ca{R}KK(Z; Cl_\tau(Z), Cl_\tau(Z)).$$
Thus, 
$[(i^{f\circ \varpi})_!]\grotimes [\wt{F_W}]$ is represented by the field
$$\{\ca{A(M)}\grotimes F_W|_z\grotimes\Cl(T_zZ),
\beta_{f\circ \varpi(z)}\grotimes\id \grotimes \gamma,
0\}_{z\in Z}\in 
\ca{R}KK(Z;\ca{S}_\vep\grotimes Cl_\tau(Z),\ca{A(M)}\grotimes Cl_\tau(Z)).$$
Similarly, $[\wt{F_W}]\grotimes [(i^{f\circ \varpi})_!]$ is also represented by the same field of Kasparov modules.

$(5)$ $[\wt{F_W}]\grotimes [d_Z]$ is the $KK$-element represented by a Dirac operator $D$ on $Z$ twisted by $F_W$, equivalently, in the unbounded picture,
$$[\wt{F_W}]\grotimes [d_Z]
=\bra{L^2(Z,\Cl(TZ)\grotimes F_W),\pi,D}.$$
By $[\wt{\varpi}^*]=[\Delta^2]^r\grotimes_{\bb{C}}\bra{
C(Z,Cl(\varpi^*TM)\grotimes S^{\fib}_{\red}),\varpi^*,0}$,
we have
$$[\widetilde{\varpi}^*]\grotimes [\wt{F_W}]\grotimes [d_Z]
=[\Delta^2]^r\grotimes_{\bb{C}}
\bra{
L^2(Z,Cl(\varpi^*TM)\grotimes S^{\fib}_{\red}\grotimes F_W),
\pi\circ \varpi^*,D}\in KK^{2r}(Cl_\tau(M),\bb{C}).$$
By \cite[Proposition 4.12]{BHS},
$S^{\fib}_{\red}\grotimes F_W$ is isomorphic to
$\Cl_{\frac{1}{2}}(T_{\fib}Z)$ defined in
\cite[Definition 3.10]{BHS}.
Moreover, by \cite[Proposition 3.11]{BHS},
the corresponding family index is represented by the trivial line bundle.
Applying \cite[Proposition 3.6]{BHS}, we have 
$$[\widetilde{\varpi}^*]\grotimes [\wt{F_W}]\grotimes [d_Z]
=[\Delta^2]^r\grotimes_{\bb{C}}
\bra{
L^2(M,\Cl(TM)),\pi,D}
=[\Delta^2]^r\grotimes_{\bb{C}}[d_M]
\in KK^{2r}(Cl_\tau(M),\bb{C}).$$
Therefore, square $(5)$ commutes.
\end{proof}

Combining the preceding results, we obtain the following.

\begin{thm}
The correspondence $\Phi$ induces a homomorphism
$K_{\ev/\odd}^{geo}(\ca{M})\to KK^{\ev/\odd}(\bb{C},\ca{A(M)})$.
\end{thm}

\subsection{A non-trivial $K$-theory element of $\ca{A(M)}$}\label{Our construction is non-trivial}

In this subsection, we prove, in certain special cases,
that the $K$-theory element constructed in the previous subsection is non-zero.
To this end, it suffices to prove that the Kasparov product of our $K$-theory element and a certain $KK$-element is non-zero.
We constructed a $K$-homology element for the following infinite-dimensional setting in \cite{T1}.
Let $T$ be the circle group, let $LT=L^2_1(S^1,T)$ be its loop group, let $(LT)_0$ denote the identity component, and let $U=(LT)_0/T$. Let $[d_U]\in KK(\ca{A}(U),\ca{S}_\vep)$ be the $KK$-element defined in \cite[Section 5.1]{T1}, where it is called the index element and is an element of the 
$U$-equivariant $KK$-group. In this paper, we forget the group action.
Let $X$ be a closed Riemannian manifold, and let $\ca{M}=U\times X$. 
In Definition \ref{pullback along dp}, we have defined $\iota_x^*:\ca{A}(U\times X)\to \ca{A}(U)\grotimes \Cl(T_xX)$ for $x\in X$.
Note that $KK^{0}(\Cl(T_xX),\Cl(T_xX))\cong \bb{Z}$. 
Let $[\bm{1}_{\Cl(T_xM)}]$ denote its generator.
Then we obtain
$$[\iota_x^*]\grotimes([d_U]\grotimes[\bm{1}_{\Cl(T_xX)}])\in
KK^{0}(\ca{A}(U\times X),\ca{S}_\vep\grotimes \Cl(T_xX)).$$
We now consider the following problem.
For simplicity, we assume $\dim(M)=\dim(X)$.

\begin{prob}\label{nontriv prob}
For a $K$-homology cycle $(M,E,f)\in K^{geo}_{\dim(M)}(\ca{M})$ such that $\dim(M)=\dim(X)$, compute 
$$\Phi(M,E,f)\grotimes [\iota_x^*]\grotimes
([d_U]\grotimes[\bm{1}_{\Cl(T_xX)}])\grotimes
[\ev_0]
\in KK^{\dim(M)}(\bb{C},\Cl(T_xX))\cong \bb{Z}.$$
\end{prob}

Intuitively, $[\iota_x^*]\grotimes([d_U]\grotimes[\bm{1}_{\Cl(T_xX)}])$ is the $K$-homology element supported on $U\times \{ x\}$ represented by the Dirac operator on $U\times \{ x\}$, while
$\Phi(M,E,f)$ plays the role of the Poincar\'e dual of the geometric cycle $(M,E,f)$ and is supported on $f(M)$.
Thus, one expects the Kasparov product to localize near the intersection of $f(M)$ and $U\times\{x\}$.

We divide the computation into several steps.
Using associativity and, where applicable, commutativity of the Kasparov product,
\begin{align*}
&\Phi(M,E,f)\grotimes [\iota_x^*]\grotimes([d_U]\grotimes[\bm{1}_{\Cl(T_xM)}])\\
&\ \ \ =
\Bigl\{
([b_+]\grotimes [E\grotimes S_M])\grotimes_{\ca{A}(M)}
[(i^f)_!]\grotimes_{\ca{A(M)}\grotimes Cl_\tau(M)}
\Bigl([\bm{1}_{\ca{A(M)}}]\grotimes_{\bb{C}}[d_M]\Bigr)\Bigr\} \\
&\ \ \ \ \ \ \grotimes_{\ca{A(M)}}
[\iota_x^*]\grotimes_{\ca{A}(U)\grotimes \Cl(T_xX)}
([d_U]\grotimes_{\bb{C}}[\bm{1}_{\Cl(T_xX)}])\\
&\ \ \ =
([b_+]\grotimes [E\grotimes S_M])\grotimes_{\ca{A}(M)}
[(i^f)_!]\grotimes_{\ca{A(M)}\grotimes Cl_\tau(M)}
\Bigl([\iota_x^*]\grotimes_{\bb{C}}[\bm{1}_{Cl_\tau(M)}]\Bigr)\\
&\ \ \ \ \ \ \grotimes
\Bigl([\bm{1}_{\ca{A}(U)\grotimes \Cl(T_xX)}]\grotimes_{\bb{C}}[d_M]\Bigr)
\grotimes_{\ca{A}(U)\grotimes \Cl(T_xX)}
([d_U]\grotimes_{\bb{C}}[\bm{1}_{\Cl(T_xX)}])\\
&\ \ \ =
([b_+]\grotimes [E\grotimes S_M])\grotimes_{\ca{A}(M)}
[(i^f)_!]\grotimes_{\ca{A(M)}\grotimes Cl_\tau(M)}
\Bigl([\iota_x^*]\grotimes_{\bb{C}}[\bm{1}_{Cl_\tau(X)}]\Bigr ) \\
&\ \ \ \ \ \ \grotimes_{\ca{A}(U)\grotimes \Cl(T_xX)\grotimes Cl_\tau(M)}
\bbra{[d_U]\grotimes_{\bb{C}}[\bm{1}_{\Cl(T_xX)}]
\grotimes_{\bb{C}}[d_M]}
\end{align*}

We first deal with $[(i^f)_!]\grotimes_{\ca{A(M)}\grotimes Cl_\tau(M)}
\Bigl([\iota_x^*]\grotimes_{\bb{C}}[\bm{1}_{Cl_\tau(X)}]\Bigr)$.
We introduce the following maps.
Let $f=(f_1,f_2):M\to U\times X$.
Let $\beta_0:\ca{S}_\vep\to \ca{A}(U)$ be the map 
$$g\mapsto g(X\grotimes 1+1\grotimes C_0),$$
where $C_0$ is the Euler vector field at the origin of $U$.
This is the restriction of the Bott map in \cite{HKT} to $\ca{S}_\vep$.
Let $(j^{f_2})_!:\ca{A}(M)\to Cl_\tau(X\times M)$ be the map
$$(j^{f_2})_!(g\grotimes \phi)(y,m):=g(C_{f_2(m)}(y))\grotimes \phi(m).$$
Unlike $(i^f)_!$, this map takes values in $Cl_\tau(X\times M)$ instead of $\ca{A}(X\times M)$.

We prove $[(i^f)_!]\grotimes_{\ca{A(M)}\grotimes Cl_\tau(M)}
\Bigl([\iota_x^*]\grotimes_{\bb{C}}[\bm{1}_{Cl_\tau(X)}]\Bigr)$ is independent of $f_1$.
Recall that $\Delta:\ca{S}_\vep\to\ca{S}_\vep\grotimes \ca{S}_\vep$ is the coproduct: $\Delta(f):=f(X\grotimes 1+1\grotimes X)$.

\begin{lem}
We have
$$[(i^f)_!]\grotimes_{\ca{A(M)}\grotimes Cl_\tau(M)}
\Bigl([\iota_x^*]\grotimes_{\bb{C}}[\bm{1}_{Cl_\tau(M)}]\Bigr)
=\bra{[\Delta]\grotimes_{\bb{C}}[\bm{1}_{Cl_\tau(M)}]}\grotimes_{\ca{S}_\vep\grotimes \ca{A}(M)}\bra{ [\beta_0]\grotimes_{\bb{C}}[ \iota_x^*\circ (j^{f_2})_!]}$$
in $KK(\ca{A}(M),\ca{A}(U)\grotimes \Cl(T_xX)\grotimes Cl_\tau(M))$.
\end{lem}
\begin{proof}
Recall that the correspondence $f\mapsto [(i^f)_!]$ is homotopy invariant.
Since $U$ is contractible, $f=(f_1,f_2)$ is homotopic to $(0,f_2)$, where $0$ means the constant map $0(m)=0\in U$, and hence $[(i^f)_!]=[(i^{(0,f_2)})_!]$.
It follows directly from the definition that $[\iota_x^*]\grotimes_{\bb{C}}[\bm{1}_{Cl_\tau(M)}]
=[\iota_x^*\grotimes \id_{Cl_\tau(M)}]$.
Thus, 
$$[(i^f)_!]\grotimes_{\ca{A(M)}\grotimes Cl_\tau(M)}
\Bigl([\iota_x^*]\grotimes_{\bb{C}}[\bm{1}_{Cl_\tau(M)}]\Bigr)
=[(\iota_x^*\grotimes \id_{Cl_\tau(M)})\circ (i^{(0,f_2)})_!].$$

Let us compute the composition 
$(\iota_x^*\grotimes \id_{Cl_\tau(X)})\circ (i^{(0,f_2)})_!$ explicitly as a map.
For $g\in\ca{S}_\vep$, let $\Delta(g)=\sum_ig_1^i\grotimes g_2^i$.
For $\phi\in Cl_\tau(M)$, and $t\in (-\vep,\vep)$, $y\in X$ and $m\in M$,
\begin{align*}
(\iota_x^*\grotimes \id_{Cl_\tau(X)})\circ (i^{(0,f_2)})_!(g\grotimes \phi)(t,u,m)
&=(i^{(0,f_2)})_!(g\grotimes \phi)(t,u,x,m) \\
&=\beta_{(0,f_2(m))}(g)(t,u,x)\grotimes \phi(m) \\
&=g(t\grotimes 1+1\grotimes C_{(0,f_2(m))}(u,x))\grotimes \phi(m) \\
&=g(\bbra{t\grotimes 1+
1\grotimes C_{0}(u)}\grotimes 1+
(1\grotimes 1)\grotimes C_{f_2(m)}(x))\grotimes \phi(m) \\
&=\sum_ig_1^i(t\grotimes 1+1\grotimes C_{0}(u))\grotimes 
g_2^i(C_{f_2(m)}(x))\grotimes \phi(m)\\
&=\sum_i\beta_0(g_1^i)(t,u)\grotimes 
g_2^i(C_{f_2(m)}(x))\grotimes \phi(m)\\
&=\beta_0\grotimes (j^{f_2})_!(\Delta(g)\grotimes\phi)(t,u,x,m) \\
&=(\beta_0\grotimes (\iota_x^*\circ (j^{f_2})_!))\circ (\Delta\grotimes\id_{Cl_\tau(M)})
(g\grotimes\phi)(t,u,m).
\end{align*}

By $[(\beta_0\grotimes (\iota_x^*\circ (j^{f_2})_!))\circ (\Delta\grotimes\id_{Cl_\tau(M)})]
=[\Delta\grotimes\id_{Cl_\tau(M)}]\grotimes [\beta_0\grotimes (\iota_x^*\circ (j^{f_2})_!)]$, we obtain the result.
Therefore, the expression $[(i^f)_!]\grotimes_{\ca{A(M)}\grotimes Cl_\tau(M)}
\Bigl([\iota_x^*]\grotimes_{\bb{C}}[\bm{1}_{Cl_\tau(M)}]\Bigr)$ depends only on $f_2$, as claimed.
\end{proof}

\begin{rmk}
Note that $(\beta_0\grotimes( \iota_x^*\circ (j^{f_2})_!))\circ (\Delta\grotimes\id_{Cl_\tau(M)})
(g\grotimes\phi)(t,u,m)$ vanishes if $d(f_2(m),x)\geq \vep$.
In this sense, it is localized near the graph of $(0,f_2)$.
Since $f$ is homotopic to $(0,f_2)$, this may be viewed as localization near the graph of $f$.
This supports the interpretation that the Kasparov product in Problem \ref{nontriv prob} computes an intersection between the graph of $f$ and $U\times \{x\}$.
\end{rmk}


In \cite[Proposition 6.3]{T1}, we have proved that $[\beta_0]\grotimes[d_U]=[\bm{1}_{\ca{S}_\vep}]$.
Thus, our Kasparov product is simplified as follows:
{\small
\begin{align*}
&\bbra{[b_+]\grotimes_{\bb{C}} [E\grotimes S_M]}
\grotimes_{\ca{A}(M)}
[\Delta\grotimes \id]
\grotimes_{\ca{S}_\vep\grotimes\ca{A}(M)}
\bbbra{\beta_0\grotimes_{\bb{C}} (\iota_x^*\circ (j^{f_2})_!)}
\grotimes_{\ca{A}(U)\grotimes \Cl(T_xX)\grotimes Cl_\tau(M)}
\bbra{[d_U]\grotimes_{\bb{C}}[\bm{1}_{\Cl(T_xX)}]
\grotimes_{\bb{C}}[d_M]} \\
&\ \ \ =\bbra{([b_+]\grotimes_{\ca{S}_\vep}[\Delta])\grotimes_{\bb{C}} [E\grotimes S_M]}
\grotimes_{\ca{S}_\vep\grotimes\ca{A}(M)}
\bbra{[\beta_0]\grotimes_{\ca{A}(U)}[d_U]}\grotimes_{\bb{C}}
\bbra{[\iota_x^*\circ (j^{f_2})_!]\grotimes_{\bb{C}}\bra{[\bm{1}_{\Cl(T_xX)}]
\grotimes_{\bb{C}}[d_M]}}\\
&\ \ \ =\bbra{([b_+]\grotimes_{\bb{C}}[b_+])\grotimes_{\bb{C}} [E\grotimes S_M]}
\grotimes_{\ca{A}(M)}
\bbra{
[\iota_x^*\circ (j^{f_2})_!]\grotimes_{\Cl(T_xX)\grotimes Cl_\tau(M)}\bra{[\bm{1}_{\Cl(T_xX)}]
\grotimes_{\bb{C}}[d_M]}}\\
&\ \ \ =
[b_+]\grotimes_{\bb{C}}
 \bra{
  \bra{[b_+]\grotimes_{\bb{C}}[E\grotimes S_M]}
 \grotimes_{\ca{A}(M)}
 [\iota_x^*\circ (j^{f_2})_!]}
\grotimes_{\Cl(T_xX)\grotimes Cl_\tau(M)}
 \bra{[\bm{1}_{\Cl(T_xX)}]
 \grotimes_{\bb{C}}[d_M]}
\end{align*}}
in $KK^{m}(\bb{C},\ca{S}_\vep\grotimes \Cl(T_xX))$.

Let us compute 
$\bra{[b_+]\grotimes_{\bb{C}}
[E\grotimes S_M]}\grotimes_{\ca{S}_\vep\grotimes\ca{A}(M)}
[\iota_x^*\circ (j^{f_2})_!]$.
The following computation is analogous to the usual computation of an intersection number in topological intersection theory.
By Sard's theorem, after deforming $f$ by a homotopy, we may assume that $x$ is a regular value of $f_2$.
We may further assume that, for each 
$y\in f_2^{-1}(x)$, the restriction of $f_2$ to a small neighborhood of $y$ is a diffeomorphism by the inverse function theorem. After possibly deforming the Riemannian metric of $M$, we may assume that this neighborhood contains the $\vep$-neighborhood of $y$, that the $\vep$-neighborhoods of the points $y\in f_2^{-1}(x)$ are mutually disjoint, and that the restriction of $f_2$ to these neighborhoods is an isometry.
We denote the inverse image of the $\vep$-neighborhood of $x$ under $f_2$ by $U_\vep$.
Since $f_2^{-1}(x)=\{y_1,y_2,\ldots,y_N\}$ is discrete, we have 
$U_\vep=\coprod_i U_\vep(y_i)$,
where $U_\vep(y_i)$ is the $\vep$-neighborhood of $y_i$ in $M$.

\begin{lem}
Let $\iota_{y_i}:\{y_i\}\hookrightarrow M$ be the inclusion.
Then, $\bra{[b_+]\grotimes_{\bb{C}}
[E\grotimes S_M]}\grotimes_{\ca{A}(M)}
[\iota_x^*\circ (j^{f_2})_!]$ is given by
$$\sum_i{\rm sign}(f_2,y_i)[\bm{1}_{\Cl(T_xX)}]\grotimes\rank(E_{y_i})(\iota_{y_i})_![\bm{1}]
\in
KK^m(\bb{C},\Cl(T_xX)\grotimes Cl_\tau(M)),$$
where ${\rm sign}(f_2,y_i)$ is $1$ if $f_2$ preserves the orientation at $y_i$, and is $-1$ if $f_2$ reverses the orientation at $y_i$.
\end{lem}
\begin{proof}
First, we clarify and modify the geometric situation.
By deforming the Riemannian metric on $M$, we may identify $U_\vep(y_i)$ with the $\vep$-ball of $\bb{R}^m$ by an orientation-preserving isometry.
Let $\phi$ denote its coordinate map.
By deforming the Riemannian metric of $X$ around $x\in X$, we may assume that $f_2$ is an isometry on each $U_\vep(y_i)$.
Then, with suitable orientation-preserving coordinates on $M$ around $y_i$ and on $X$ around $x$,
we may assume $f_2$ is given near $y_i$ by one of the following matrices:
\begin{equation}\label{df_2 normalize}
\begin{pmatrix}
1 & 0 & \cdots & 0 \\
0 & 1 & \cdots & 0 \\
\vdots & \vdots & \ddots & \vdots \\
0 & 0 & \cdots & 1
\end{pmatrix}
\text{ or }
\begin{pmatrix}
-1 & 0 & \cdots & 0 \\
0 & 1 & \cdots & 0 \\
\vdots & \vdots & \ddots & \vdots \\
0 & 0 & \cdots & 1
\end{pmatrix}\tag{$*$}
\end{equation}
according as $f_2$ preserves or reverses the orientation at $y_i$, respectively.
We identify $T_{y_i}M$ with $\bb{R}^m$ via the derivative $d\phi$ of the coordinate map $\phi$.

We now describe the relevant Kasparov modules in the above local geometric setting.
$$[\iota_x^*\circ (j^{f_2})_!]
=\bra{\Cl(T_xX)\grotimes Cl_\tau(M),\iota_x^*\circ (j^{f_2})_!,0
}\in KK^0(\ca{A}(M),\Cl(T_xX)\grotimes Cl_\tau(M)),\text{ and}$$
$$[b_+]\grotimes_{\bb{C}}
[E\grotimes S_M]=
\bra{\ca{S}_\vep\grotimes C(M,E\grotimes S_M),\id\grotimes \gamma,\frac{X}{\vep}\grotimes\id}
\in KK^m(\bb{C},\ca{A}(M)),$$
where $\gamma$ is the left action of $\Cl(\bb{R}^m)$ induced by the multigraded $Spin^c$-structure of $S_M$.
Thus, in order to compute their Kasparov product, we first need to identify
$$\bbra{\ca{S}_\vep\grotimes C(M,E\grotimes S_M)}
\grotimes_{\ca{A}(M),\iota_x^*\circ (j^{f_2})_!}
\bbra{\Cl(T_xX)\grotimes Cl_\tau(M)}.$$
We prove that it is isomorphic to $\Cl(T_xX)\grotimes C_0(U_\vep,E\grotimes S_M)$.
By 
$$C(M,E\grotimes S_M)\cong C(M,E\grotimes S_M)\grotimes_{Cl_\tau(M)} Cl_\tau(M),$$
we have a homomorphism
\begin{align*}
&\bbra{\ca{S}_\vep\grotimes C(M,E\grotimes S_M)}
\grotimes_{\ca{A}(M),\iota_x^*\circ (j^{f_2})_!}
\bbra{\Cl(T_xX)\grotimes Cl_\tau(M)}\\ 
&\ \ \ \xrightarrow{\cong}
C(M,E\grotimes S_M)\grotimes_{Cl_\tau(M)}
\ca{A}(M)
\grotimes_{\ca{A}(M),\iota_x^*\circ (j^{f_2})_!}
\bbra{\Cl(T_xX)\grotimes Cl_\tau(M)} \\
&\ \ \ \ \ \ \to 
C(M,E\grotimes S_M)\grotimes_{Cl_\tau(M)}
\bbra{\Cl(T_xX)\grotimes Cl_\tau(M)} 
\xrightarrow{\cong}
 \Cl(T_xX)\grotimes C(M,E\grotimes S_M).
\end{align*}
More explicitly, the homomorphism is given as follows.
For $g\in \ca{S}_\vep$, $s\in C(M,E\grotimes S_M)$, $a\in \Cl(T_xX)$, and $\phi\in Cl_\tau(M)$, 
\begin{align*}
&(g\grotimes s)\grotimes(a\grotimes \phi)
\mapsto 
(-1)^{|g||s|}s\grotimes(g\grotimes 1)\grotimes (a\grotimes \phi)
\mapsto
(-1)^{|g||s|}s\grotimes\iota_x^*\circ (j^{f_2})_!(g\grotimes 1)(a\grotimes \phi)\\
&\ \ \ = 
[M\ni m\mapsto 
(-1)^{|g||s|+(|g|+|a|)|s|}g(C_{f_2(m)}(x))a\grotimes s(m)\phi(m)]
\\
&\ \ \ =
[M\ni m\mapsto 
(-1)^{|a||s|}g(C_{f_2(m)}(x))a\grotimes s(m)\phi(m)].
\end{align*}
Thanks to the factor $g(C_{f_2(m)}(x))$, it vanishes if $d(f_2(m),x)\geq \vep$, and hence it is supported on $U_\vep=\cup_{y\in f_2^{-1}(x)}U_\vep(y)$. Thus, we have
$$\bbra{\ca{S}_\vep\grotimes C(M,E\grotimes S_M)}
\grotimes_{\ca{A}(M),\iota_x^*\circ (j^{f_2})_!}
\bbra{\Cl(T_xX)\grotimes Cl_\tau(M)}
\cong \Cl(T_xX)\grotimes C_0(U_\vep,E\grotimes S_M).$$

We next describe this Kasparov module on each component $U_\varepsilon(y_i)$.
Since each $U_\vep(y_i)$ is contractible, we may choose a trivialization of $S_M$ over $U_\vep(y_i)$ and identify it with the standard one: $S_M|_{U_\vep(y_i)}=U_\vep(y_i)\times \Cl(T_{y_i}M)$.
Then the fiber is isomorphic to
$\Cl(T_xX)\grotimes E_{y_i}\grotimes \Cl(T_{y_i}M)$.
In this description, the multigrading structure is given by $d\phi^{-1}:\bb{R}^m\to T_{y_i}M$.

Under this isomorphism, the operator 
$\bra{\frac{X}{\vep}\grotimes \id_{E\grotimes S_M}}\grotimes \id$ is given by multiplication by the function
$U_\vep\ni z\mapsto F_z:=\frac{1}{\vep}C_{f_2(m)}(x)\grotimes \id_{E_{y_i}}\grotimes \id_{S_M}$.
Therefore, 
$[b_+]\grotimes_{\bb{C}}
[E\grotimes S_M]$ is represented by
\begin{equation}\label{[b_+]grotimes_{bb{C}}
[Egrotimes S_M]sonomono}
\bigoplus_i\bra{C_0(U_\vep(y_i))\grotimes \Cl(T_xX)\grotimes E\grotimes \Cl(T_{y_i}M),\id_{\Cl(T_xX)}\grotimes \id_{E_{y_i}}\grotimes (\gamma\circ d\phi^{-1}),\bbra{F_z}_{z\in U_\vep(y_i)}}.\tag{$**$} 
\end{equation}

Note that $\Cl(\bb{R}^m)$, which gives the multigrading structure, acts on $\Cl(T_{y_i}M)$, and this action is denoted by $\gamma\circ d\phi^{-1}$.
On the other hand, $F_z$ is defined using the $\Cl(T_xX)$-factor.
These two operators can be represented by a single Clifford multiplication
$${\small\xymatrix{
\widetilde{\gamma}:\bb{R}^m\oplus T_{y_i}M \ar[r]
\ar@{}[d]|{\rotatebox{-90}{$\in$}} &
T_xX\oplus T_{y_i}M \ar[r]
\ar@{}[d]|{\rotatebox{-90}{$\in$}} &
\Cl(T_xX\oplus T_{y_i}M) \ar^{\cong}[r]
\ar@{}[d]|{\rotatebox{-90}{$\in$}} &
 \Cl(T_xX)\grotimes \Cl( T_{y_i}M)
 \\
(v,z) \ar@{|->}[r]&
\bra{-\frac{df_2(m)}{\vep}, d\phi^{-1}(v)} \ar@{|->}^{=}[r]&
\bra{-\frac{df_2(m)}{\vep}, d\phi^{-1}(v)}.
}}$$
More precisely, $(\gamma\circ d\phi^{-1})(v)\grotimes \id=\widetilde{\gamma}(v,0)$, and $F_z=\id\grotimes \widetilde{\gamma}(0,\log_{y_i}(z))$ as operators on $\Cl(T_xX\oplus T_{y_i}M)$. 
We would like to ``rotate'' the action in order to split the $\Cl(T_xX)$-factor from the $Cl_\tau(U_\vep(y_i))$-factor.
Consider the continuous path in the orthogonal group $O(T_xX\oplus T_{y_i}M)$ given by 
$$t\mapsto 
\begin{pmatrix}
(\cos t)\id_{T_xX} & (\sin t) df_2 \\
-(\sin t) df_2^{-1} & (\cos t)\id_{T_{y_i}M}
\end{pmatrix}
\in O(T_xX\oplus T_{y_i}M).$$
By the above homotopy, the map 
$(v,z) \mapsto
\bra{-\frac{df_2(m)}{\vep}, d\phi^{-1}(v)}$
is homotopic to the map obtained at $t=\frac{\pi}{2}$, namely,
$$\bb{R}^m\oplus T_{y_i}M
\ni
(v,z)
\mapsto
\bra{df_2\circ d\phi^{-1}(v), \frac{z}{\vep}}
\in T_xX\oplus T_{y_i}M.$$
Composing with $\widetilde{\gamma}$, 
the pair consisting of the $\Cl(\bb{R}^m)$-action 
$\id\grotimes (\gamma\circ d\phi^{-1})$ and 
the operator $F_z\grotimes \id$, both on 
$\Cl(T_xX) \grotimes \Cl(T_{y_i}M)$, is homotopic 
to the pair consisting of $(\gamma\circ df_2\circ d\phi^{-1})\grotimes \id$ and
$\id\grotimes \frac{z}{\vep}$.

Therefore, the Kasparov module (\ref{[b_+]grotimes_{bb{C}}
[Egrotimes S_M]sonomono}) is homotopic to 
\begin{align*}
&\bigoplus_i\bra{
\Cl(T_xX)\grotimes C_0(U_\vep(y_i))\grotimes  
E_{y_i}\grotimes \Cl(T_{y_i}M),
\gamma\circ df_2\circ d\phi^{-1}\grotimes
\id,
\bbra{\id\grotimes \frac{z}{\vep}}_{z\in U_\vep(y_i)}}\\
&\ \ \ 
\cong
\bra{\Cl(T_xX),\gamma\circ df_2\circ d\phi^{-1},0}
\grotimes_{\bb{C}}
\bigoplus_i\bra{
C_0(U_\vep(y_i))\grotimes  
E_{y_i}\grotimes \Cl(T_{y_i}M),
\id,
\bbra{\id\grotimes \frac{z}{\vep}}_{z\in U_\vep(y_i)}}\\
&\ \ \ \ \ \ 
\cong
\bra{\Cl(T_xX),\gamma\circ df_2\circ d\phi^{-1},0}
\grotimes_{\bb{C}}
\bigoplus_i\bra{
C_0(U_\vep(y_i))\grotimes \Cl(T_{y_i}M),
\id,
\bbra{\frac{z}{\vep}}_{z\in U_\vep(y_i)}}^{\oplus\rank(E_{y_i})}.
\end{align*}

By (\ref{df_2 normalize}), 
$\bra{\Cl(T_xX),\gamma\circ df_2\circ d\phi^{-1},0}$
is $[\bm{1}_{\Cl(T_xX)}]$ if $df_2$ is orientation-preserving at $y_i$, and $-[\bm{1}_{\Cl(T_xX)}]$ if $df_2$ is orientation-reversing at $y_i$.
In both cases, we have 
$$\bra{\Cl(T_xX),\gamma\circ df_2\circ d\phi^{-1},0}={\rm sign}(f,y_i)[\bm{1}_{\Cl(T_xX)}].$$
Moreover, $\bra{
C_0(U_\vep(y_i))\grotimes \Cl(T_{y_i}M),
\id,\bbra{\frac{z}{\vep}}_{z\in U_\vep(y_i)}}$ is the image of $[\bm{1}]\in KK(\bb{C},\bb{C})=K^0(\{y_i\})$ under the Gysin map $(\iota_{y_i})_!$.
\end{proof}

Since $\bra{
C_0(U_\vep(y_i))\grotimes \Cl(T_{y_i}M),
\id,\bbra{\frac{z}{\vep}}_{z\in U_\vep(y_i)}}
\cong \bra{Cl_\tau(U_\vep(y_i)),
\id,\bbra{\frac{z}{\vep}}_{z\in U_\vep(y_i)}}$, and 
\\
$\bra{C_0(U_\vep(y_i),\Cl(TM)),
1,\bbra{\frac{z}{\vep}}_{z\in U_\vep(y_i)}}\grotimes_{Cl_\tau(M)} [d_M]=1$ as proved in \cite[Theorem 4.8]{Kas88}, the above computation gives the following result.

\begin{pro}
In the above situation, if $M$ is connected,
$$\Phi(M,E,f)\grotimes [\iota_x^*]\grotimes([d_U]\grotimes[\bm{1}_{\Cl(T_xX)}])\grotimes[\ev_0]$$
is equal to the mapping degree of $f_2$ times the rank of $E$.
\end{pro}

\begin{rmks}
$(1)$ If $\dim(M)<\dim(X)$, we can deform $f$ so that $f_2(M)$ and the $\vep$-neighborhood of $x$ are disjoint, and hence $\Phi(M,E,f)\grotimes [\iota_x^*]\grotimes([d_U]\grotimes[\bm{1}_{\Cl(T_xX)}])\grotimes[\ev_0]=0$.
It may be possible to prove $\Phi(M,E,f)\neq 0$ by replacing $[d_U]\grotimes[\bm{1}_{\Cl(T_xX)}]$ with $[d_U]\grotimes[\iota^*]\grotimes[d_{X'}]$ for some suitably chosen submanifold $X'$ with inclusion $\iota:X'\hookrightarrow X$.
Strictly speaking, $\iota^*$ gives a map $Cl_\tau(X)\to C(X',\Cl(TX|_{X'}))
=C(X',\Cl(TX')\grotimes \Cl(\nu))$, where $\nu$ is the normal bundle.
Thus, one needs a $Spin^c$-structure on $\nu$ to define $[d_{X'}]$.
In this setting, if $\dim(X')=\dim(X)-\dim(M)$, the computation proceeds almost in parallel with the one above.
If $\dim(X')<\dim(X)-\dim(M)$, the Kasparov product must be zero for the same reason. If $\dim(X')>\dim(X)-\dim(M)$, 
the Kasparov product can be computed in a way parallel to that in the following paragraph.

$(2)$ If $\dim(M)>\dim(X)$, one needs to deal with a more complicated inverse image $f_2^{-1}(x)$, which is a submanifold when $x$ is a regular value. 
In this case, $E|_{f_2^{-1}(x)}$ can be non-trivial, and the computation of the Kasparov product will require the arguments involving the characteristic classes of $E$.
\end{rmks}

\section*{Acknowledgements}
I am supported by JSPS KAKENHI Grant Number 23K12970.
During the preparation of this manuscript, I used ChatGPT (OpenAI) for English-language polishing and readability improvements. I reviewed and edited the resulting text and takes full responsibility for the manuscript.

Doman Takata, 
Faculty of Education Mathematical and Natural Sciences,
Niigata University, 
8050 Ikarashi 2-no-cho, Nishi-ku, Niigata, 950-2181, Japan. 

E-mail address: {\tt d.takata@ed.niigata-u.ac.jp}

\end{document}